\documentclass[12pt,a4paper, final]{amsart}
\usepackage{amssymb}
\usepackage{hyperref}
\usepackage[notcite, notref]{showkeys}
\usepackage[utf8]{inputenc}
\usepackage{stmaryrd}
\usepackage{graphicx}
\usepackage{blkarray} 
\usepackage{tikz}
\usetikzlibrary{shapes.geometric, arrows.meta, positioning}

\usepackage{color}

\theoremstyle{plain}
\newtheorem{theorem}{Theorem}[section]
\newtheorem{proposition}[theorem]{Proposition}
\newtheorem{corollary}[theorem]{Corollary}
\newtheorem{lemma}[theorem]{Lemma}

\theoremstyle{definition}

\newtheorem{remark}[theorem]{Remark}

\newtheorem{conjecture}[theorem]{Conjecture}

\theoremstyle{remark}

\numberwithin{equation}{section}

\newcommand{\R}{\mathbb R}
\newcommand{\C}{\mathbb C}

\newcommand{\fg}{\mathfrak g}
\newcommand{\fh}{\mathfrak h}

\newcommand{\fk}{\mathfrak k}

\newcommand{\fm}{\mathfrak m}

\DeclareMathOperator{\SO}{SO}
\DeclareMathOperator{\SU}{SU}

\newcommand{\sll}{\mathfrak{sl}}

\newcommand{\su}{\mathfrak{su}}

\newcommand{\op}{\operatorname}
\newcommand{\Id}{\textup{Id}}

\newcommand{\mi}{\mathrm{i}}

\DeclareMathOperator{\tr}{Tr}
\DeclareMathOperator{\diag}{diag}
\DeclareMathOperator{\Span}{Span}
\newcommand{\inner}[2]{\langle {#1},{#2}\rangle }
\newcommand{\innerdots}{\langle {\cdot},{\cdot}\rangle }

\DeclareMathOperator{\Hom}{Hom}
\DeclareMathOperator{\End}{End}

\DeclareMathOperator{\Ad}{Ad}
\DeclareMathOperator{\ad}{ad}

\DeclareMathOperator{\Spec}{Spec}

\DeclareMathOperator{\Scal}{Scal}

\DeclareMathOperator{\Ric}{Ric}

\DeclareMathOperator{\vol}{vol}

\newcommand{\kil}{\operatorname{B}}

\newcommand{\abc}{(a,b,c)}

\title[Hodge Laplacian on $1$-forms]{Hodge Laplacian on $1$-forms of homogeneous $3$-spheres}

\author{Jonas~Henkel}
\address{Fachbereich Mathematik und Informatik,
	Philipps-Universit\"at Marburg,
	Campus Lahnberge,
	35032 Marburg, Germany.}
\email{henkelj@mathematik.uni-marburg.de}

\author{Emilio~A.~Lauret}
\address{Instituto de Matemática (INMABB), Departamento de Matemática, Universidad Nacional del Sur (UNS)-CONICET, Bahía Blanca, Argentina.}
\email{emilio.lauret@uns.edu.ar}

\thanks{The second named author was supported by grants from SGCYT--UNS (PGI 24/L126, 24/ZL25) and CONICET (PIP 11220210100343CO)}
\date{\today}

\begin{document}

\begin{abstract}
We study the spectrum of the Hodge-Laplacian on $1$-forms for left-invariant metrics on the Lie group $\operatorname{SU}(2) \cong S^3$ and its quotient $\operatorname{SO}(3)\cong P^3(\mathbb{R})$. 
To the best of our knowledge, we provide the first explicit computation of the full spectrum of the Hodge-Laplacian for a canonical variation by determining the eigenvalues of Berger 3-spheres and analyzing their resulting splitting behavior.
Furthermore, we propose and rigorously prove an explicit formula for the first eigenvalue of general homogeneous metrics on $\operatorname{SU}(2)$ and $\operatorname{SO}(3)$. The formal proof of this result was autonomously discovered by an advanced AI model, providing a notable case study for AI-driven mathematical research.
Finally, leveraging this explicit formula, we apply these spectral results to the inverse problem, showing that the spectrum on $1$-forms determines the metric up to isometry.
The source code for the symbolic computations, visualizations, and a Monte Carlo stress test is provided in the electronic supplementary material \cite{HL2025Code}.
\end{abstract}
\maketitle
	
\tableofcontents
	
\section{Introduction}

The spectrum of the Hodge-Laplacian on a compact Riemannian manifold encodes fundamental geometric and topological information. A central theme in spectral geometry is to understand the extent to which the spectrum determines the geometry of the underlying manifold. While the computation of the spectrum is well understood for Riemannian symmetric spaces, much less is known in the general homogeneous setting, particularly for differential forms of degree $p > 0$.

For compact Riemannian symmetric spaces, the problem is essentially solved using representation theory. Following the foundational work of Matsushima and Murakami \cite{MatsushimaMurakami-HarmonicForms}, who established the identification of the Hodge-Laplacian with the Casimir operator for symmetric spaces of non-compact type, Ikeda and Taniguchi \cite{IkedaTaniguchi-SnPnC} adapted this method to the compact setting. Based on this algebraic description, they performed the explicit computation of the full spectrum and eigenforms for spheres and complex projective spaces. 
Subsequently, the spectra of various other symmetric spaces were studied (see e.g.\ \cite{Tsukamoto81}, \cite{Sthanumoorthy84}, \cite{Mashimo97},  \cite{BenHalima-Grassmannian}, \cite{Chami04}, \cite{Chami12}), as well as some locally symmetric spaces (see e.g.\  \cite{Ikeda88}, \cite{GornetMcGowan06},  \cite{MR-survey}, \cite{LauretMiatelloRossetti-LensSpaces}, \cite{Lauret-p-spectralens}, \cite{GornetMcGowan20}).

In the non-symmetric case, the situation for the Laplace-Beltrami operator ($p=0$) is much better understood. Bérard-Bergery and Bourguignon \cite{BerardBergeryBourguignon-LaplSubmersion} derived general results for Riemannian submersions with totally geodesic fibers. A key feature in this setting is that the Laplacian on the total space commutes with the vertical Laplacian. This allows for a simultaneous diagonalization, where the spectrum of a canonical variation $g_{(t_0, t_1)}$ (obtained by rescaling the metric on the fibers $F_p$ by $t_1$ and on the horizontal distribution by $t_0$) is determined by the eigenvalues of the original metric and the eigenvalues of the fiber:
\[ \Spec(M,g_{(t_{0},t_{1})}) \subset \{t^{-1}_{0}\lambda_{0}+(t^{-1}_{1}-t^{-1}_{0}) \eta_{1} \mid \lambda_{0}\in \Spec(M,g),\; \eta_{1}\in \Spec(F_{p},g|_{F_{p}})\}. \]
Recently, a method using spherical representations was developed in \cite{AgricolaHenkel} to explicitly determine which pairs of eigenvalues occur in the spectrum of a large class of canonical variations. 
For Berger spheres on $S^3$, Tanno~\cite{Tanno79} computed the full spectrum explicitly (see also \cite[Thm.~5.5]{Lotay12}). 
Later, the first eigenvalue for general homogeneous metrics on $S^3$ was determined in \cite{Lauret-SpecSU(2)}. 
For higher dimensional homogeneous spheres see \cite{BettiolPiccione13a}, \cite{BLPhomospheres}, \cite{BLPfullspec}. 

However, for $p$-forms with $p > 0$, this structure breaks down and no analogous formula seems to exist in general, as discussed extensively in \cite{GilkeyLeahyPark-book}. Furthermore, Gilkey and Park \cite{GilkeyPark-RiemSubmersion} proved that eigenforms of the base manifold can only be lifted to eigenforms of the total space if the horizontal distribution is integrable. For the Hopf fibration $S^3 \to S^2$, the horizontal distribution is non-integrable, which further complicates the analysis.
Regarding the asymptotic behavior, Colbois and Courtois \cite[Thm.~1.1]{Colbois} established criteria for the first eigenvalue to vanish on manifolds collapsing along a torus action. Applied to the Hopf fibration on $S^3$, their results imply that the first non-zero eigenvalue of $\Delta_1$ tends to zero as the fiber volume vanishes \cite[Ex.~1.2]{Colbois}.

In contrast to the limited results for the Hodge-Laplacian, the spectrum of the Dirac operator on $S^{3}$ equipped with the Berger metrics has been successfully computed by Bär~\cite{Baer-DiracLensSpaces}. 
More recently, Kling and Schueth~\cite{KlingSchueth-DiracSpec} determined the smallest Dirac eigenvalue on $S^{3}$ for general homogeneous metrics of positive scalar curvature and proved that the Dirac spectrum determines the isometry class of the metric.

In this paper, we address the mentioned gap by explicitely describing the spectrum of the Hodge-Laplacian on 1-forms for the family of Berger spheres on $S^3 \cong \SU(2)$ and its quotient $P^3(\R)\cong \SO(3)$. These metrics arise as canonical variations of the Hopf fibration and serve as a fundamental test case for Riemannian submersions with totally geodesic fibers beyond the symmetric setting.
Our main contributions are as follows:
\begin{enumerate}
	\item We derive a general formula for the spectrum of the Hodge-Laplacian on compact homogeneous Riemannian manifolds (Proposition \ref{prop: spectrum general}), reducing the spectral problem to a combination of representation theory (branching rules) and linear algebra (diagonalization of infinitely many matrices).
	\item We apply this to derive an explicit formula for the full spectrum of $\Delta_1$ on Berger spheres (Theorem \ref{thm4:Spec(Berger)}). We observe a splitting behavior of the eigenvalues, which differs significantly from the scalar case ($p=0$). We give a geometric interpretation of this splitting in Remark \ref{rem: interpretation Berger}. The electronic supplementary material \cite{HL2025Code} provides the source code for the symbolic computation of the eigenvalues in the general homogeneous case (for low weights $k$) and visualizes the full spectrum of the Berger metrics. 
We confirm that the first eigenvalue tends to zero in the collapse limit, consistent with \cite[Ex.~1.2]{Colbois}.

\item We state the First Eigenvalue Conjecture for general homogeneous metrics in $S^3$ and $P^3(\R)$
(Conjecture \ref{conj:first_eigenvalue}), asserting that the minimum is attained either on the spectrum of functions or on the trivial representation (left-invariant 1-forms). This conjecture is derived from explicit calculations and a two-tiered numerical stress test utilizing dynamic Gershgorin bounds \cite{HL2025Code}. Subsequently, we provide a rigorous analytical proof of the conjecture. The core of this proof was autonomously discovered by an advanced AI model.
\item We apply our explicit formula for the first eigenvalue to the inverse spectral problem. This allows us to prove that the spectrum of $\Delta_1$ completely determines the isometry class of the metric (Theorem \ref{thm:spectral_determination_positive} and Corollary \ref{cor:non-positive scalar curvature isospectral}).
\end{enumerate}

\subsection*{{Autonomous AI resolution of the First Eigenvalue Conjecture \& AI methodology}}

\begin{figure}[b]
	\centering
	\begin{tikzpicture}[
		box/.style={rectangle, draw=black, thick, align=center, rounded corners, minimum height=0.9cm, text width=3.8cm, font=\small\linespread{0.9}\selectfont},
		mainbox/.style={rectangle, draw=black, thick, align=center, rounded corners, minimum height=1cm, text width=7.5cm, font=\bfseries},
		arrow/.style={-{Stealth[scale=1.2]}, thick}
		]
		
		\node[box] (low_k) at (-2.3, 2.5) {Precise Spectrum for \\ Low Weights ($k \le 1$)};
		\node[box] (gershgorin) at (2.3, 2.5) {Asymptotic Growth \\ ($\mathcal{O}(k^2)$ Gershgorin)};
		
		\node[box] (berger) at (-4.5, 0.8) {Full Spectrum \\ of Berger Spheres};
		\node[box] (monte_carlo) at (4.5, 0.8) {Monte Carlo Stress Test \\ (Generic parameters)};
		
		\node[mainbox] (conjecture) at (0, -1) {First Eigenvalue Conjecture};
		
		\node[mainbox] (proof) at (0, -2.8) {Rigorous Proof via Curl Operator \\ (Autonomous AI Discovery)};
		
		\draw[arrow] (low_k) -- (conjecture.110);
		\draw[arrow] (gershgorin) -- (conjecture.70);
		\draw[arrow] (berger) -- (conjecture.160);
		\draw[arrow] (monte_carlo) -- (conjecture.20);
		
		\draw[arrow] (conjecture) -- (proof);
		
	\end{tikzpicture}
\caption{The research workflow leading to the desired result. The formulation of the First Eigenvalue Conjecture builds on explicit low-weight spectra, asymptotic matrix bounds, and numerical stress tests. The formal proof subsequently bypasses the representation theoretical setup entirely by employing the Curl operator.}
	\label{fig:workflow}
\end{figure}
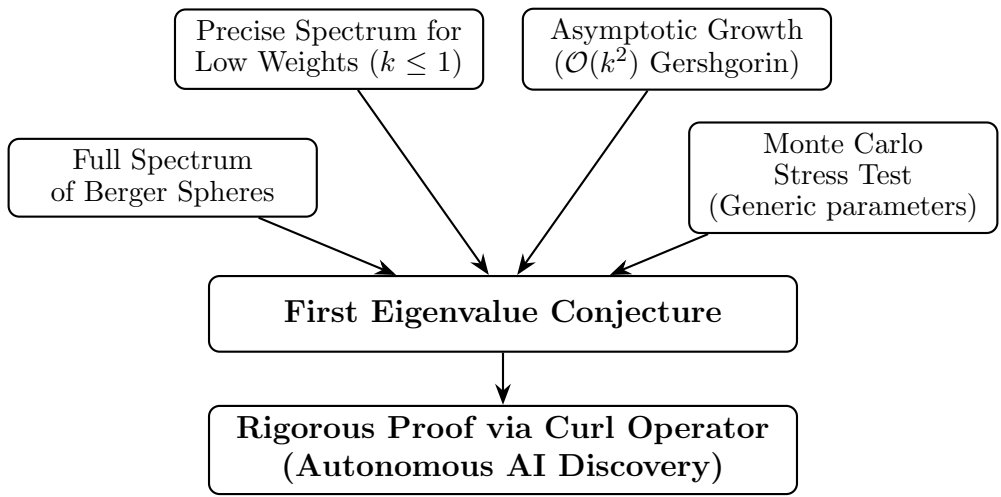
%
In the following, we outline the process that led to the formulation of the First Eigenvalue Conjecture and detail the AI-driven discovery of its rigorous proof (summarized in Figure \ref{fig:workflow}). Furthermore, we clarify the general use of AI-assisted technologies in the broader preparation of this manuscript.
	
	While the representation theoretical approach developed in Section \ref{sec:Hodge SU(2)} made it possible to compute the complete Hodge spectrum for Berger metrics, the authors were unable to determine the exact global minimum from the resulting matrices for general metrics. 

By analyzing the explicit spectra for low weights ($k=0,1$), the asymptotic $\mathcal{O}(k^2)$ growth of the Gershgorin bounds, and the exact solution for the Berger sphere sub-family, we identified a strong structural candidate for the global minimum. We verified this candidate against asymmetrical metrics using Monte Carlo simulations \cite{HL2025Code}, culminating in the precise formulation of the First Eigenvalue Conjecture.

The rigorous mathematical proof of this conjecture was discovered by the AI model ChatGPT 5.4 Pro. Operating completely autonomously for 100 minutes within a compute-intensive tier (200 USD/month, limited to 15 queries), the model chose a fundamentally different approach. Rather than choosing a representation theoretic viewpoint, it identified a geometric shortcut: factoring the Hodge Laplacian through the Curl operator. While this factorization holds on any manifold where the dimension $n$ and form degree $p$ satisfy $n=2p+1$, it is highly effective for $1$-forms in $3$ dimensions. 

By applying this transformation, the AI bypassed the algebraic framework entirely and reduced the problem to an operator relation on the round sphere. The authors verified this logic and adapted it to the paper's notation. {{Since the model did not provide citations and its internal reasoning process (Chain of Thought) remained opaque, we independently traced the conceptual roots of its strategy back to the existing literature to establish the scholarly context.}} This workflow reflects the emerging research practices discussed in \cite{Henkel-AI-Math}.

\subsubsection{Declaration of Generative AI and AI-assisted technologies in the preparation of this manuscript}
The authors acknowledge the use of the AI model ChatGPT 5.4 Pro, which autonomously discovered and formulated the core analytical proof strategy for the First Eigenvalue Conjecture (Section \ref{sec:proof_conjecture}). Additionally, the model Grok 4 was used for assistance with algebraic verifications in the proof of Proposition~\ref{prop:k1_eigenvalue_comparison}. Furthermore, the first-named author used generative AI tools (Gemini 2.5/3 Pro, ChatGPT 4o-5.5 Thinking, DeepL) to assist with the electronic supplementary material (cf.~\cite{HL2025Code}), language refinement, literature search and analysis. After using these tools, the authors reviewed and edited the content as needed and take full responsibility for the content of the publication.

	\subsection*{Acknowledgments}
	The authors would like to thank Ilka Agricola for her helpful discussions on the manuscript. The first author is particularly grateful for her support and advice as his PhD supervisor, which made his visit to Córdoba (Argentina) possible and paved the way for this collaboration. His stay was funded by the Deutsche Forschungsgemeinschaft (DFG, German Research Foundation) - 541920696.
	
	This project started during the joint workshop ``Lie Theory in Geometry,
	Algebra and Analysis'' held in Córdoba in March 2025. Both authors
	gratefully acknowledge the hospitality of the National University of
	Córdoba during this event. Furthermore, the second author thanks the
	University of Marburg for their hospitality during his visit in
	September 2025, when a substantial part of this work was carried out.
	
	Finally, the authors are grateful to Dorothee Schüth for pointing out
	the reference \cite{BransonGilkeyOrsted-HeatInv} regarding the heat
	invariants.

	\section{Preliminaries}
	
	\subsection{Hodge-Laplace operators on compact homogeneous spaces}
	Let $G$ be a compact Lie group of dimension $m$, $K$ a closed subgroup of $G$, and we set $M=G/K$. 
	We denote by $n$ the dimension of $M$. 
	We fix a decomposition at the Lie algebra level $\fg=\fk\oplus \fm$, where $\fm$ is $\Ad_{\restriction K}$-invariant. 
	Any $G$-invariant Riemannian metric $g$ on $M$ is induced by an $\Ad_{\restriction K}$-invariant inner product $\innerdots$ on $\fm$. 
	
	Let $p$ be an integer satisfying $0\leq p\leq n$. 
	The action $\Ad_{\restriction K}$ of $K$ on $\fm$ induces an action on $\bigwedge^{p}\fm^{*}_{\C}$, denoted by $\tau_p$. 
	Its associated homogeneous vector bundle is precisely the $p$-exterior cotangent bundle, that is, 
	\begin{align*}
		{\textstyle \bigwedge^{p}}\, T^*M
		= G\underset{\tau_p}{\times} {\textstyle \bigwedge^{p}} \fm^{*}_{\C}
		:=\Big(G\times {\textstyle \bigwedge^{p}} \fm^{*}_{\C}\Big)/\sim,
	\end{align*}
	where $(x,\alpha)\sim (xk,\tau_p(k^{-1})\cdot \alpha)$ for every $x\in G$, $k\in K$, and $\alpha\in {\textstyle \bigwedge^{p}} \fm^{*}_{\C}$. 
	We denote by $[x,\alpha]$ the class of $(x,\alpha)\in G\times {\textstyle \bigwedge^{p}} \fm^{*}_{\C}$ in $G{\times_{\tau_p}} {\textstyle \bigwedge^{p}} \fm^{*}_{\C}$. 
	The fiber at $xK\in M$ is $\{[x,\alpha]: \alpha\in {\textstyle \bigwedge^{p}} \fm^{*}_{\C}\}$.

	The set of smooth sections of the vector bundle ${\textstyle \bigwedge^{p}}\, T^{*}(M)$ is precisely the space of differential $p$-forms $\Omega^p(M)$. 
	We identify $\Omega^p(M)$ with 
	$$
	\mathcal{C}^{\infty}_{K}(G,{\textstyle \bigwedge^{p}} \fm^{*}_{\C}):=\{f:G\to {\textstyle \bigwedge^{p}} \fm^{*}_{\C}: \text{$f$ is smooth and }f(xk)=\tau_{p}(k^{-1})\cdot f(x)\quad\forall \, x\in G,k\in K\}
	$$ 
	as follows: 
	if $\omega\in\Omega^p(M)$, then  $\omega_{xK}:=\omega(xK)=[x,f_{\omega}(x)]$ for some
	$f_\omega \in \mathcal{C}^{\infty}_{K}(G,{\textstyle \bigwedge^{p}} \fm^{*}_{\C})$, and similarly, 
	if $f\in \mathcal{C}^{\infty}_{K}(G,{\textstyle \bigwedge^{p}} \fm^{*}_{\C})$, then 
	$xK\mapsto [x,f(x)]$ defines a $p$-form $\omega_f$. 
	
	The group $G$ acts on ${\textstyle \bigwedge^{p}}\, T^{*}(M)$ by $a\cdot [x,\alpha]=[ax,\alpha]$. 
	Furthermore, $G$ acts on $\Omega^p(M)$ by $(a\cdot \omega)(xK)= a\cdot \omega(a^{-1}xK)$. 
	This action corresponds with the action of $G$ on $\mathcal{C}^{\infty}_{K}(G,{\textstyle \bigwedge^{p}} \fm^{*}_{\C})$ given by $(a\cdot f)(x)=f(a^{-1}x)$. 
	
	From now on, we fix $g$ a $G$-invariant metric on $M$, so $(M,g)$ is a compact homogeneous Riemannian manifold. 
	
	The Hodge-Laplacian $\Delta_p$ associated to $(M,g)$ acting on $p$-forms is given by $\Delta_p=d\delta+\delta d$, where $\delta$ stands for the co-differential of $(M,g)$. 
	The Weitzenböck formula (see e.g.\ \cite{SemmelmannWeingart-QK}, \cite[Thm.~2.38]{Rosenberg-Laplacian}) gives $\Delta_p= \nabla^*\nabla +2q(R)$. If $\{E_i\}_{i=1}^n$ is a local orthonormal frame, $\{E_i^{*}\}_{i=1}^n$ is its dual coframe and $Y,Z_{1},...,Z_{p}$ are vector fields, the components of the Weitzenböck formula are given by
	\begin{equation}\label{eq:generalHodgeLaplacian-definition}
		\begin{aligned}
			\nabla^*\nabla \omega&
			=-\sum_{i=1}^n \nabla_{E_i} \big(\nabla_{E_i} \omega\big)-\nabla_{(\nabla_{E_{i}}E_{i})}\omega,\quad 
			\nabla_{Y}\omega 
			= Y(\omega) +\nabla^{\fm}_{Y}(\omega)
			\\ 
			\nabla^{\fm}_{Y}(\omega)(Z_{1},\dots,Z_{p})&
			=- \sum_{i=1}^p \omega(Z_1,\dots,Z_{i-1},\nabla_YZ_i,Z_{i+1},\dots,Z_p)
			,\\
			(q(R)\omega)_{aK}
			&=\frac{1}{2}\sum_{i<j}(E_{i}\wedge E_{j})\triangleright R(E_{i}\wedge E_{j})\triangleright \omega_{aK}, \quad 
			R(X \wedge Y): = \frac{1}{2} \sum_{i} E_i \wedge R(X,Y)E_i
			\\
			(E_{i}\wedge E_{j})\triangleright \omega_{aK}&
			=E_{j}^{*}\wedge (E_{i}\lrcorner\omega_{aK})-E_{i}^{*}\wedge (E_{j}\lrcorner\omega_{aK}),	
		\end{aligned}
	\end{equation}
	for any $\omega\in \Omega^p(G/K)\simeq \mathcal{C}^{\infty}_{K}(G,\bigwedge^{p}\fm^{*}_{\C})$.
	The definition of $\nabla^{*}\nabla$ and the definition of $2q(R)$ are based on tensors, hence they are $\mathcal{C}^{\infty}(M)$-polynomial expressions in the orthonormal frame $\{E_{i}\}_{i=1}^{n}$ \cite[Prop.~4.17, 4.21]{Lee-IRM}, i.e.\ the formula does not depend on the extension of $\{E_{i}(aK)\}_{i=1}^{n}$.
	In order to evaluate the Weitzenböck formula in the origin only, let $\{X_i\}_{i=1}^{n}$ be any orthonormal basis of $\fm$ with respect to $\innerdots$, and $\{X_i^*\}_{i=1}^{n}$ be its dual coframe. By abuse of notation, we denote the corresponding fundamental vector fields which yield an orthonormal basis only in $eK$ in the same way. Since $G$ is compact and the decomposition is reductive, it is a standard fact that $\sum \nabla_{X_i}X_i = 0$ at the origin (see e.g.\ \cite[p.~183, 184]{Besse-Einstein}). Hence, the Weitzenböck formula reduces in the origin $eK$ to
	\begin{equation}\label{eq:HodgeLaplacian-definition}
		\begin{aligned}
			(\Delta_{p}\omega)_{eK}= \sum (\nabla_{X_i}\nabla_{X_i}\omega)|_{eK}+\frac{1}{2}\sum_{i<j}((X_{i}\wedge X_{j})\triangleright R(X_{i}\wedge X_{j})\triangleright \omega)|_{eK}.
		\end{aligned}
	\end{equation}
	
	For any irreducible representation $(\varrho,V_{\varrho})$ of $G$ over the complex numbers, we have the embedding $V_\pi\otimes \Hom_K(V_{\varrho}, {\textstyle\bigwedge^p} \fm_\C) \hookrightarrow \mathcal C_K^\infty(G, {\textstyle\bigwedge^p} \fm_\C )\equiv \Omega^p(M)$ given by $v\otimes L\mapsto \omega_{v\otimes L}$, where 
	\begin{equation*}
		\omega_{v\otimes L}(xK)= L\big(\varrho(x^{-1})v\big). 
	\end{equation*}
	The Hilbert space given by the completion of $C_K^\infty(G, {\textstyle\bigwedge^p} \fm_\C )$ with respect to the inner product induced by $\innerdots$ is denoted by $L^2({\textstyle\bigwedge^p} T^*M)$. 
	Peter and Weyl Theorem decomposes it as (see e.g.\ \cite[Thm.~1.3]{Takeuchi} for functions; \cite[Thm.~5.3.6]{Wallach-HarmonicAnalysis} for general homogeneous vector bundles)
	\begin{equation}\label{eq:PeterWeyl}
		L^2\big({\textstyle\bigwedge^p} T^*M \big)
		= \widehat{\bigoplus_{\varrho\in\widehat G}} \; V_{\varrho}\otimes \Hom_K(V_{\varrho}, {\textstyle\bigwedge^p} \fm_\C^{*} ), 
	\end{equation}
	where the sum at the right is the Hilbert sum of Hilbert spaces. 
	
	The Hodge-Laplacian $\Delta_p$ leaves invariant each isotypical component $V_{\varrho}\otimes \Hom_K(V_{\varrho}, {\textstyle\bigwedge^p} \fm_\C^{*} )$. 
	The next goal is to give an expression of it for the general case.
	
	Via the canonical isomorphism $V_{\varrho}^*\otimes {\textstyle \bigwedge^{p}}\fm^{*}_{\C}\simeq \Hom(V_{\varrho}, {\textstyle \bigwedge^{p}}\fm^{*}_{\C})$, $\varphi\otimes \alpha\longmapsto \big(v\mapsto \varphi(v)\alpha\big)$, the $K$-equivariant elements corresponds to 
	\begin{align*}
		\left(V_{\varrho}^{*}\otimes {\textstyle \bigwedge^{p}}\fm^{*}_{\C} \right)^{K}
		= \Span_\C\left\{
		\varphi \otimes \alpha \mid  
		(\varrho^*( k^{-1})\cdot \varphi)\otimes \alpha=\varphi\otimes (\tau_p(k)\cdot \alpha) \;\;\forall k\in K
		\right\}
		\subset V_{\varrho}^{*}\otimes {\textstyle \bigwedge^{p}}\fm^{*}_{\C}
		.
	\end{align*}

	For any Lie algebra $\fh$ equipped with an inner product, let $\{Z_k\}$ be an orthonormal basis and $d\varrho$ be an irreducible representation of $\fh$. We set $C_\varrho^\fh := \sum_k d\varrho(Z_k)^2$ and note that this operator depends on the inner product, but not on the choice of the orthonormal basis. 
	It coincides with the image of the classical Casimir element under $\varrho$ if and only if the inner product is $\ad(\fh)$-invariant. In this case, by Schur's Lemma, it acts as a scalar on any irreducible representation, computable via Freudenthal's formula.
	
	The Levi-Civita connection induces a linear map $\nabla: \fg \to \End(\bigwedge^p \fm_\C^*)$, which is in general not a Lie algebra representation. In analogy to the above, we define the operator $C_\nabla^{\fm} := \sum_k (\nabla_{X_k})^2$, where $\{X_k\}$ is an orthonormal basis of $\fm$.
	\begin{lemma}\label{lem: formula Hodge Laplace general}
		Let $\varrho\in\widehat G$. 
		The Hodge-Laplace operator $\Delta_{p}$ preserves $V_{\varrho}\otimes \Hom_{K}(V_{\varrho}, \bigwedge^{p}\fm^{*}_{\C}) $ in \eqref{eq:PeterWeyl}.
		Moreover, for $v\in V_{\varrho}$ and $L\in \Hom_{K}(V_{\varrho}, \bigwedge^{p}\fm^{*}_{\C})$, $\Delta_p\cdot \omega_{v\otimes L} = \omega_{v\otimes (\Delta_p^{(\varrho)}\cdot L)}$, where 
		\begin{align*}
			&\Delta_{p}^{(\varrho)}\restriction \left(V_{\varrho}^{*}\bigotimes\bigwedge^{p}\fm^{*}_{\C}\right)^{K}\\
			& = -((C^{\fg}_{\varrho^*} -C^{\fk}_{\varrho^*})\otimes \Id_{\tau_p} 
			+ \Id_{\varrho^*}\otimes C^{\fm}_{\nabla} 
			+ \sum_{i=1}^{n}2 d\varrho^*(X_{i})\otimes \nabla^{\fm}_{X_{i}})
			+ \Id_{\varrho^*}\otimes 2q(R)
.
		\end{align*}
	\end{lemma}

	\begin{proof}
		We identify the space of $p$-forms $\Omega^p(M)$ with the space of smooth $K$-equivariant functions $\mathcal{C}^{\infty}_{K}(G, \bigwedge^{p} \fm^{*}_{\C})$. A $p$-form $\omega$ is thus treated as a function $\omega: G \to \bigwedge^{p} \fm^{*}_{\C}$.
		
First, we establish how the natural action of $G$ on these functions translates to the Peter-Weyl decomposition. 
Since the action of $a \in G$ on a $p$-form $\omega$ is given by $(a \cdot \omega)(x) = \omega(a^{-1}x)$,
we have for any $v \in V_\varrho$ and $L \in \Hom_K(V_\varrho, \bigwedge^p \fm_\C^*)$ that
		\begin{align*}
			(a \cdot \omega_{v \otimes L})(x) &= \omega_{v \otimes L}(a^{-1}x) = L(\varrho((a^{-1}x)^{-1})v) \\
			&= L(\varrho(x^{-1}a)v) = L(\varrho(x^{-1})(\varrho(a)v)) \\
			&= \omega_{(\varrho(a)v) \otimes L}(x).
		\end{align*}
		This shows that the action of $G$ on $\mathcal{C}^{\infty}_{K}(G, \bigwedge^{p} \fm^{*}_{\C})$ corresponds precisely to the standard action of $G$ on the first factor $V_\varrho$ in the tensor product $V_\varrho \otimes \Hom_K(V_\varrho, \bigwedge^p \fm_\C^*)$.
		
		The Hodge-Laplacian $\Delta_p$ is a natural geometric operator constructed from the $G$-invariant metric $g$. As such, it commutes with the action of $G$. By Schur's Lemma, its action on each isotypical component must be of the form $\Id_{V_\varrho} \otimes \Delta_p^{(\varrho)}$ for some linear operator $\Delta_p^{(\varrho)} \in \End(\Hom_K(V_\varrho, \bigwedge^p \fm_\C^*))$. Our goal is to determine this operator by evaluating the action of $\Delta_p$ on a form $\omega_{v \otimes L}$ at the identity $e \in G$.
		
		The covariant derivative $\nabla_X$ for a fundamental vector field $X \in \fm$ splits at the identity into a tangential and a fiber part:
		\begin{equation*}
			(\nabla_X \omega)(e) = X(\omega)(e) + \nabla_X^\fm(\omega)(e),
		\end{equation*}
		where $X(\omega)$ denotes the directional derivative. 
We next compute these parts for $\omega_{v \otimes L}$. 
We have that
		\begin{align*}
			X_i(\omega_{v \otimes L})(e) 
			&= \left.\frac{d}{dt}\right|_{t=0} \omega_{v \otimes L}(\exp(tX_i)) 
			= \left.\frac{d}{dt}\right|_{t=0} L(\varrho(\exp(tX_i)^{-1})v) \\
			&= \left.\frac{d}{dt}\right|_{t=0} L(\varrho(\exp(-tX_i))v) = L(-d\varrho(X_i)v),
\\
\nabla_{X_i}^\fm(\omega_{v \otimes L})(e) &= \nabla_{X_i}(L(v)).
		\end{align*}
Now we assemble the connection Laplacian term from \eqref{eq:HodgeLaplacian-definition}:
		\begin{align*}
			-\sum_{i=1}^n (\nabla_{X_i} \nabla_{X_i} \omega_{v \otimes L})(e) = -\sum_{i=1}^n \Big( X_i(X_i(\omega)) + X_i(\nabla_{X_i}^\fm \omega) + \nabla_{X_i}^\fm(X_i(\omega)) + \nabla_{X_i}^\fm(\nabla_{X_i}^\fm \omega) \Big)(e).
		\end{align*}
		We identify each term with an operator acting on $L$:
		\begin{enumerate}
			\item $\sum_i X_i(X_i(\omega_{v \otimes L}))(e) = L( \sum_i (d\varrho(X_i))^2 v)$. This corresponds to the operator $L \mapsto L \circ (\sum_i d\varrho(X_i)^2)$.
			
			\item $\sum_i \nabla_{X_i}^\fm(\nabla_{X_i}^\fm \omega_{v \otimes L})(e) = \sum_i (\nabla_{X_i})^2 (L(v))$. This corresponds to $L \mapsto (\sum_i (\nabla_{X_i})^2) \circ L$.
			
			\item The cross terms are $\sum_i (X_i(\nabla_{X_i}^\fm \omega) + \nabla_{X_i}^\fm(X_i(\omega)))$. At the identity, both evaluate to $\nabla_{X_i}^\fm(X_i(\omega_{v \otimes L}))(e) = \nabla_{X_i}(L(-d\varrho(X_i)v))$. The sum corresponds to the operator $L \mapsto -2 \sum_i \nabla_{X_i} \circ L \circ d\varrho(X_i)$.
		\end{enumerate}
		Combining all parts with the curvature term from \eqref{eq:HodgeLaplacian-definition}, we find that $(\Delta_p \omega_{v \otimes L})(e)$ is the result of applying the operator $\Delta_p^{(\varrho)}$ to $L$ and then evaluating at $v$. The operator $\Delta_p^{(\varrho)}$ is given by:
		\begin{equation*}
			\Delta_p^{(\varrho)}(L) =-( L \circ \left(\sum_{i=1}^n d\varrho(X_i)^2\right) + \left(\sum_{i=1}^n (\nabla_{X_i})^2\right) \circ L - 2 \sum_{i=1}^n \nabla_{X_i} \circ L \circ d\varrho(X_i)) + 2q(R) \circ L.
		\end{equation*}
		This formula describes the operator on $\Hom_K(V_\varrho, \bigwedge^p \fm_\C^*)$. By identifying this space with $(V_\varrho^* \otimes \bigwedge^p \fm_\C^*)^K$, the actions by composition translate into actions on the tensor factors. Specifically, $L \circ A$ corresponds to an action on the first factor $V_\varrho^*$, and $B \circ L$ corresponds to an action on the second factor $\bigwedge^p \fm_\C^*$. This yields precisely the expression given in the lemma.
	\end{proof}
	The Peter-Weyl theorem \eqref{eq:PeterWeyl} ensures that the entire space of $p$-forms is a Hilbert sum of its isotypical components.
	Each $\varrho$-isotypical component is isomorphic to the tensor product $V_\varrho \otimes \Hom_K(V_\varrho, \bigwedge^p \fm_\C^*)$. Since the Hodge-Laplacian acts as $\Id_{V_\varrho} \otimes \Delta_p^{(\varrho)}$ on this space, its global spectrum is simply the union of the eigenvalues of the finite-dimensional operators $\Delta_p^{(\varrho)}$ acting on the intertwiner spaces $\Hom_K(V_\varrho, \bigwedge^p \fm_\C^*)$. The following proposition formalizes this conclusion.
\begin{proposition}\label{prop: spectrum general}
		The spectrum of the Hodge-Laplace operator $\Delta_{p}$ on $(G/K,g)$ is given by the union of eigenvalues of the operators 
		\begin{align*}
			\Delta_{p}^{(\varrho)}=-( (C^{\fg}_{\varrho^*} -C^{\fk}_{\varrho^*})\otimes \Id_{\tau_p} 
			+ \Id_{\varrho^*}\otimes C_{\nabla}^{\fm} 
			+ \sum_{i=1}^{n}2 d\varrho^*(X_{i})\otimes \nabla_{X_{i}}^{\fm})
			+ \Id_{\varrho^*}\otimes 2q(R),
		\end{align*}
		where $\varrho\in \widehat{G} $ such that $\Hom_{K}(V_{\varrho},\bigwedge^p \fm_\C^*) \neq 0$. The multiplicities are given by $\dim(V_{\varrho})$.
	\end{proposition}
	\begin{remark}
		Proposition \ref{prop: spectrum general} illustrates a general method for computing the spectrum of any homogeneous differential operator on a vector bundle over $G/K$. 
Spectral analysis on homogeneous spaces always boils down to two main tasks: 
\begin{enumerate}
			\item  A representation-theoretic branching problem to identify which irreducible representations $\varrho \in \widehat{G}$ appear in the space of sections.
			\item A linear algebra problem to compute the eigenvalues of the operator induced by the differential operator on each of the corresponding finite-dimensional intertwining spaces.
		\end{enumerate}
		
	 In the specific setting of this paper, where $G=\SU(2)$ and $K=\{\text{Id}\}$, the branching problem is trivial, as every irreducible representation appears. Consequently, the difficulty lies entirely in the second task: solving the linear algebra problem for the Hodge-Laplacian, which involves non-trivial interaction terms in the Weitzenböck formula.
	\end{remark}

	The general formula in Lemma~\ref{lem: formula Hodge Laplace general} simplifies dramatically in the classical setting of a Riemannian symmetric space, yielding the following well-known result.
	
	\begin{corollary}[Spectrum on Symmetric Spaces]
		If $(G/K,g)$ is a compact Riemannian symmetric space, the operator $\Delta_p^{(\varrho)}$ acts as a scalar on the intertwiner space $\Hom_K(V_\varrho, \bigwedge^p \fm_\C^*)$. This scalar is given by the Casimir eigenvalue $c_\fg(\varrho)$.
		Consequently, the spectrum of the Hodge-Laplacian $\Delta_p$ is the discrete set
		\begin{equation*}
			\Spec(\Delta_p) = \{ c_\fg(\varrho) \mid \varrho \in \widehat{G} \text{ such that } \Hom_K(V_\varrho, \bigwedge^p \fm_\C^*) \neq 0 \},
		\end{equation*}
		where each eigenvalue has multiplicity $\dim(V_\varrho) \cdot \dim(\Hom_K(V_\varrho, \bigwedge^p \fm_\C^*))$.
Moreover, if $\varrho$ has highest weight $\Lambda$, then
		\begin{equation*}
			c_\fg(\varrho) = \inner{\Lambda}{\Lambda + 2\delta},
		\end{equation*}
		where $\delta$ is half the sum of the positive roots and $\inner{\cdot}{\cdot}$ is the inner product on the dual of a Cartan subalgebra induced by the Killing form (up to normalization).
	\end{corollary}
	
	\begin{proof}
		For a symmetric space, several key simplifications occur in the formula for $\Delta_p^{(\varrho)}$.
		First, the connection term $\nabla_X$ vanishes for any $X \in \fm$ when evaluated at the origin. This is because $(\nabla_X^{\fm} Y)_{eK} = \frac{1}{2}[X,Y]_{eK}$ for fundamental vector fields on a normal homogeneous space, and for a symmetric space, $[\fm,\fm] \subseteq \fk$, which implies $[X,Y]_{\fm} = 0$. The vanishing of $\nabla_X^{\fm}$ for all $X \in \fm$ implies that the mixed term $2\sum d\varrho^*(X_i) \otimes \nabla_{X_{i}}^{\fm}$ and the term involving $\sum_i (\nabla_{X_i}^{\fm})^2$ are both zero. Additionally, the operators $C_{\varrho^*}^\fg$, $C_{\tau_p}^\fk$ act as scalars computable by Freudenthal's formula. The formula reduces to
		\begin{equation*}
			\Delta_p^{(\varrho)} = (c_\fg(\varrho^*) - c_\fk(\varrho^*)) \otimes \Id_{\tau_p} + \Id_{\varrho^*} \otimes 2q(R).
		\end{equation*}
		Second, it is a standard result for symmetric spaces that the curvature operator is given by the action of the Casimir operator of $\fk$, i.e., $2q(R) = -C_{\tau_p}^\fk$, see for example \cite[Lem. B.0.11]{Semmelmann-Habil}. Substituting this in, we get
		\begin{equation*}
			\Delta_p^{(\varrho)} = (c_\fg(\varrho^*) - c_\fk(\varrho^*)) \otimes \Id_{\tau_p} - \Id_{\varrho^*} \otimes C_{\tau_p}^\fk.
		\end{equation*}
The action of the Casimir of $\fk$ on the first factor must coincide with its action on the second. That is, the action of $c_\fk(\varrho^*) \otimes \Id_{\tau_p}$ is identical to that of $\Id_{\varrho^*} \otimes C_{\tau_p}^\fk$. The two terms involving $C^\fk$ thus cancel, leaving
		\begin{equation*}
			\Delta_p^{(\varrho)} = c_\fg(\varrho^*) \otimes \Id_{\tau_p}.
		\end{equation*}
		Since the Casimir eigenvalues of a representation and its dual are identical ($c_\fg(\varrho^*) = c_\fg(\varrho)$), the operator acts as the scalar $c_\fg(\varrho)$. 
The statement on the spectrum and multiplicities follows directly from Proposition~\ref{prop: spectrum general}. 
	\end{proof}

	\section{Hodge-Laplace operators on homogeneous 3-spheres}\label{sec:Hodge SU(2)}
	We now apply the general framework from the previous section to our main case of interest: homogeneous 3-spheres. 
	
	\subsection{The homogeneous space $(\SU(2), g_{\abc})$: setup and notation}\label{subsec:SU(2)}
	
	In this subsection, we assemble the specific geometric and representation-theoretic components needed for our analysis. Our primary model is the Lie group $G = \SU(2)$, which is diffeomorphic to the 3-sphere $S^3$. The results also apply to $G = \SO(3) \simeq \SU(2)/\{\pm\Id\}$, which is diffeomorphic to the real projective space $P^3(\R)$. We consider the case $K=\{\Id\}$, so $\fg=\fm=\su(2)$.
	\subsubsection{Homogeneous metrics on 3-spheres}
	We set 
	\begin{align}\label{eq2:X1X2X3}
		E_1&= \begin{pmatrix} i&0 \\ 0&-i \end{pmatrix}, &
		E_2&= \begin{pmatrix} 0&1 \\ -1&0 \end{pmatrix}, &
		E_3&= \begin{pmatrix} 0&i \\ i&0 \end{pmatrix}.
	\end{align}
	It is a simple matter to check that $\{E_1,E_2,E_3\}$ is an orthonormal basis of $\su(2)$ with respect to the $\Ad(\SU(2))$-invariant inner product $\langle X,Y\rangle_0 := -\frac18\kil_{\fg}(X,Y)= -\frac12\op{tr}(XY)$ for $X,Y\in\su(2)$.
	Here, $\kil_\fg(\cdot,\cdot)$ denotes the Killing form on $\fg\otimes_\R\C\simeq \sll(2,\C)$. 
	Furthermore, $[E_1,E_2]=2E_3$, $[E_3,E_1]=2E_2$ and $[E_2,E_3]=2E_1$. 
	
	For $a,b,c$ positive real numbers, let $\innerdots_{\abc}$ denote the inner product on $\fg$ such that $\{aE_1,bE_2,cE_3\}$ is an orthonormal basis, and let $g_{\abc}$ be the corresponding left-invariant metric on $G$. 
	Any permutation of $\abc$ does not change the isometry class of $g_{\abc}$. 
	Moreover, any homogeneous metric on $G$ is isometric to $g_{\abc}$ for some $a,b,c$ (see e.g.\ \cite{Milnor-Curvatures}, \cite[Rem. 2.2]{KlingSchueth-DiracSpec}). 
	Write $X_1=aE_1$, $X_2=bE_2$, and $X_3=cE_3$, thus $\{X_1,X_2,X_3\}$ is an orthonormal basis of $\innerdots_{\abc}$.
	Furthermore, we denote by $\{X_1^*, X_2^*, X_3^*\}$ the basis of $\fg^*$ dual to $\{X_1,X_2,X_3\}$. 
	
	\subsubsection{The Levi-Civita connection}
	The Levi-Civita connection of the metric $g_{\abc}$ can be computed using the Koszul formula. For left-invariant vector fields $X,Y,Z \in \su(2)$, it states
	\begin{equation*}
		2\inner{\nabla_X Y}{Z} = \inner{[X,Y]}{Z} - \inner{[Y,Z]}{X} + \inner{[Z,X]}{Y}.
	\end{equation*}
	As an example, we compute the components of $\nabla_{X_1}X_2$. 
	We have that
	\begin{align*}
		2\inner{\nabla_{X_1} X_2}{X_3} &= \inner{[X_1,X_2]}{X_3} - \inner{[X_2,X_3]}{X_1} + \inner{[X_3,X_1]}{X_2} \\
		&= \inner{\tfrac{2ab}{c}X_3}{X_3} - \inner{\tfrac{2bc}{a}X_1}{X_1} + \inner{\tfrac{2ca}{b}X_2}{X_2} 
		= 2(\tfrac{ab}{c} - \tfrac{bc}{a} + \tfrac{ca}{b}).
	\end{align*}
	The other components are zero, e.g., $2\inner{\nabla_{X_1} X_2}{X_1} = \inner{[X_1,X_2]}{X_1} - \dots = 0$. 
It follows that $\nabla_{X_1}X_2 = (\frac{ab}{c} - \frac{bc}{a} + \frac{ac}{b})X_3$. The remaining terms are computed analogously, resulting in the following formulas:
	\begin{equation}\label{eq:christoffel}
		\begin{aligned}
\nabla_{X_1}X_2 &= (\tfrac{ab}{c}+\tfrac{ac}{b}-\tfrac{bc}{a})\,X_3, & 
\qquad
\nabla_{X_2}X_1 &= (-\tfrac{ab}{c}+\tfrac{ac}{b}-\tfrac{bc}{a})\,X_3, 
\\
\nabla_{X_1}X_3 &= (-\tfrac{ab}{c}-\tfrac{ac}{b}+\tfrac{bc}{a})\,X_2, & 
\nabla_{X_3}X_1 &= (-\tfrac{ab}{c}+\tfrac{ac}{b}+\tfrac{bc}{a})\,X_2, 
\\
\nabla_{X_2}X_3 &= (+\tfrac{ab}{c}-\tfrac{ac}{b}+\tfrac{bc}{a})\,X_1, & 
\nabla_{X_3}X_2 &= (+\tfrac{ab}{c}-\tfrac{ac}{b}-\tfrac{bc}{a})\,X_1.
		\end{aligned}
	\end{equation}
	As usual, we denote the Christoffel symbols by  $\Gamma_{i,j}^k=X_k^*(\nabla_{X_i}X_j)$.
	\subsubsection{The Curvature Operator}
	We can view the Riemannian curvature tensor $R$ as a symmetric operator on the space of 2-vectors, $R: \Lambda^2\su(2) \to \Lambda^2\su(2)$, defined by
	\begin{equation}
		R(X \wedge Y) = \frac{1}{2} \sum_{i} X_i \wedge R(X,Y)X_i.
	\end{equation}
	For the metric $g_{\abc}$, this operator is diagonal with respect to the orthonormal basis $\{X_1 \wedge X_2, X_1 \wedge X_3, X_2 \wedge X_3\}$. We demonstrate the computation for the first diagonal entry $r_{12}$, defined by $R(X_1 \wedge X_2) = r_{12} (X_1 \wedge X_2)$. Expanding the definition yields the relation
	\begin{equation*}
		r_{12} = \inner{R(X_1 \wedge X_2)}{X_1 \wedge X_2} = -\inner{R(X_1, X_2)X_2}{X_1}.
	\end{equation*}
	We first compute the vector field term $R(X_1, X_2)X_2$. 
Using $[X_1, X_2] = \frac{2ab}{c}X_3$, we have that
\begin{align*}
R(X_1, X_2)X_2 
	&= \nabla_{X_1} \nabla_{X_2} X_2 - \nabla_{X_2} \nabla_{X_1} X_2 - \nabla_{[X_1, X_2]} X_2 
\\ &
	= 0 - \Gamma_{12}^3 \nabla_{X_2} X_3 - \tfrac{2ab}{c} \,\Gamma_{32}^1 X_1 
	= - (\Gamma_{12}^3 \Gamma_{23}^1 + \tfrac{2ab}{c} \Gamma_{32}^1) X_1.
\end{align*}
Substituting this back into the expression for $r_{12}$ gives
$
r_{12} = \Gamma_{12}^3 \Gamma_{23}^1 + \tfrac{2ab}{c} \Gamma_{32}^1.
$
	Substituting the expressions for the Christoffel symbols from \eqref{eq:christoffel} and simplifying leads to
	\begin{equation*}
		r_{12} = 3\frac{a^2b^2}{c^2} - \frac{a^2c^2}{b^2} - \frac{b^2c^2}{a^2} - 2a^2 - 2b^2 + 2c^2.
	\end{equation*}
	The other eigenvalues $r_{13}$ and $r_{23}$ are obtained by cyclic permutation of the indices and parameters. The curvature operator is thus given by
	\begin{equation}\label{eq:curvature_eigenvalues}
		R = \diag(r_{12}, r_{13}, r_{23}),
	\end{equation}
	with
\begin{align*}
r_{13} &= -\frac{a^2b^2}{c^2} + 3\frac{a^2c^2}{b^2} - \frac{b^2c^2}{a^2} - 2a^2 + 2b^2 - 2c^2,
\\
r_{23} &= -\frac{a^2b^2}{c^2} - \frac{a^2c^2}{b^2} + 3\frac{b^2c^2}{a^2} + 2a^2 - 2b^2 - 2c^2.
\end{align*}
	
	\subsubsection{Representation theory}
	The irreducible unitary representations of $\SU(2)$ are indexed by non-negative integers $k \in \mathbb{N}_0$. We denote the $(k+1)$-dimensional representation by $(\varrho_k, V_k)$. A standard model for $V_k$ is the space of homogeneous polynomials of degree $k$ in two complex variables. Throughout this paper, we will use the basis $\{P_l : 0 \le l \le k\}$ for $V_k$ given by $P_l(z,w)=z^lw^{k-l}$. 
	The action of the Lie algebra basis $\{E_1, E_2, E_3\}$ on these basis vectors is given by (see the proof of Lemma~3.1 in \cite{Lauret-SpecSU(2)})
	\begin{equation}\label{eq:rep_action}
		\begin{aligned}
			d\varrho_k(E_1) \cdot P_l &
			= (k-2l)\mi \, P_l
			,\\
			d\varrho_k(E_2) \cdot P_l &
			= -l\, P_{l-1} + (k-l)\, P_{l+1}
			,\\
			d\varrho_k(E_3) \cdot P_l &= -l\mi\, P_{l-1} - (k-l)\mi\, P_{l+1}.
		\end{aligned}
	\end{equation}
	Here we adopt the convention that $P_l = 0$ if $l < 0$ or $l > k$.
	
	\subsection{Computation of the Hodge-Laplacian matrix}
	
	With the representation-theoretic and geometric data at hand, we can now construct the matrix for the Hodge-Laplacian $\Delta_1$ on each isotypic component $V_k \otimes \su(2)_\C^*$. We begin by computing the constant, fiber-wise acting part of the operator.
	
	\begin{lemma}\label{lem:curvature_operator}
		The operator $2q(R)$ acts diagonally on the dual basis $\{X_1^*, X_2^*, X_3^*\}$. Its action is given by
		\begin{equation}\label{eq:curvature_term}
			\begin{aligned}
				2q(R)(X_1^*) &= 2\left(-\frac{a^{2}b^{2}}{c^{2}}-\frac{a^{2}c^{2}}{b^{2}}+\frac{b^{2}c^{2}}{a^{2}}+2a^{2}\right) X_1^*, \\
				2q(R)(X_2^*) &= 2\left(-\frac{a^{2}b^{2}}{c^{2}}-\frac{b^{2}c^{2}}{a^{2}}+\frac{a^{2}c^{2}}{b^{2}}+2b^{2}\right) X_2^*, \\
				2q(R)(X_3^*) &= 2\left(-\frac{b^{2}c^{2}}{a^{2}}-\frac{a^{2}c^{2}}{b^{2}}+\frac{a^{2}b^{2}}{c^{2}}+2c^{2}\right) X_3^*.
			\end{aligned}
		\end{equation}
	\end{lemma}
	
	\begin{proof}
		The curvature operator $R:\Lambda^{2}\su(2)\to\Lambda^{2}\su(2)$ is diagonal with respect to the basis $\{X_{1}\wedge X_{2}, X_{1}\wedge X_{3}, X_{2}\wedge X_{3}\}$. 
Its eigenvalues 
are (cf.~\eqref{eq:curvature_eigenvalues})
		\begin{align*}
			r_{12} &= 3\tfrac{a^2b^2}{c^2} - \tfrac{a^2c^2}{b^2} - \tfrac{b^2c^2}{a^2} - 2a^2 - 2b^2 + 2c^2, \\
			r_{13} &= -\tfrac{a^2b^2}{c^2} + 3\tfrac{a^2c^2}{b^2} - \tfrac{b^2c^2}{a^2} - 2a^2 + 2b^2 - 2c^2, \\
			r_{23} &= -\tfrac{a^2b^2}{c^2} - \tfrac{a^2c^2}{b^2} + 3\tfrac{b^2c^2}{a^2} + 2a^2 - 2b^2 - 2c^2.
		\end{align*}
		The operator $2q(R)$ is defined as the sum of the actions of the following components:
		\begin{equation*}
			2q(R) = r_{12} (X_1 \wedge X_2) \triangleright (X_1 \wedge X_2) \triangleright 
			+ r_{13} (X_1 \wedge X_3) \triangleright (X_1 \wedge X_3) \triangleright 
			+ r_{23} (X_2 \wedge X_3) \triangleright (X_2 \wedge X_3) \triangleright.
		\end{equation*}
		A direct computation shows that the action of a simple 2-vector on the basis $\{X_k^*\}$ is given by
		\begin{equation*}
			(X_i \wedge X_j) \triangleright X_k^* = \delta_{ik}X_j^* - \delta_{jk}X_i^*,
		\end{equation*}
		and applying the action twice yields
		\begin{equation*}
			(X_i \wedge X_j) \triangleright (X_i \wedge X_j) \triangleright X_k^* = -(\delta_{ik} + \delta_{jk})X_k^*.
		\end{equation*}
		This shows that $2q(R)$ is diagonal in the basis $\{X_k^*\}$. The eigenvalue corresponding to $X_1^*$ is obtained by summing the contributions from the three terms, that is, 
		\begin{equation*}
			2q(R) (X_1^*) = r_{12}(-X_1^*) + r_{13}(-X_1^*) + r_{23}(0) = (-r_{12} - r_{13}) X_1^*.
		\end{equation*}
		The eigenvalues for $X_2^*$ and $X_3^*$ are computed analogously to be $(-r_{12} - r_{23})$ and $(-r_{13} - r_{23})$, respectively. Substituting the expressions for $r_{ij}$ yields the formulas stated in the lemma. We note that these eigenvalues are precisely the eigenvalues of the classical Ricci operator $\Ric$, which is for $\Delta_{1}$ confirmed by the literature, see \cite[Thm.~2.44]{Rosenberg-Laplacian}.
	\end{proof}

	\begin{lemma}\label{lem:connection_casimir}
		The operator $C_\nabla^{\su(2)} = \sum_{k=1}^3 \nabla_{X_k}^2$ acts diagonally on the dual basis $\{X_1^*, X_2^*, X_3^*\}$. Its action is given by
		\begin{align*}
			C_\nabla^{\su(2)}(X_1^*) &= -2\left(\frac{b^2c^2}{a^2} + \frac{a^2c^2}{b^2} + \frac{a^2b^2}{c^2} - 2a^2\right) X_1^*, \\
			C_\nabla^{\su(2)}(X_2^*) &= -2\left(\frac{b^2c^2}{a^2} + \frac{a^2c^2}{b^2} + \frac{a^2b^2}{c^2} - 2b^2\right) X_2^*, \\
			C_\nabla^{\su(2)}(X_3^*) &= -2\left(\frac{b^2c^2}{a^2} + \frac{a^2c^2}{b^2} + \frac{a^2b^2}{c^2} - 2c^2\right) X_3^*.
		\end{align*}
	\end{lemma}
	
	\begin{proof}
		We compute the action of $C_\nabla^{\su(2)}$ on each basis covector $X_j^*$. The action of a covariant derivative on a dual basis vector is given by $(\nabla_{X_k} X_j^*)(X_l) = -X_j^*(\nabla_{X_k} X_l) = -\Gamma_{k,l}^j$. This implies
		\begin{equation*}
			\nabla_{X_k} X_j^* = -\sum_{l=1}^3 \Gamma_{k,l}^j X_l^*=\sum_{l=1}^3 \Gamma_{k,j}^l X_l^*.
		\end{equation*}
		We now compute $C_\nabla^{\su(2)}(X_1^*) = (\nabla_{X_1}^2 + \nabla_{X_2}^2 + \nabla_{X_3}^2)(X_1^*)$.
		First, $\nabla_{X_1}(X_1^*) = -(\Gamma_{1,1}^1 X_1^* + \Gamma_{1,2}^1 X_2^* + \Gamma_{1,3}^1 X_3^*)$. From \eqref{eq:christoffel}, we see that $\nabla_{X_1}X_j$ never has an $X_1$ component, so $\Gamma_{1,j}^1 = 0$ for all $j$. Thus, $\nabla_{X_1}(X_1^*) = 0$, and consequently $\nabla_{X_1}^2(X_1^*) = 0$.
		Secondly, $\nabla_{X_2}(X_1^*) = -(\Gamma_{2,1}^1 X_1^* + \Gamma_{2,2}^1 X_2^* + \Gamma_{2,3}^1 X_3^*) = -\Gamma_{2,3}^1 X_3^*$, since only $\Gamma_{2,3}^1$ is non-zero.
		Applying $\nabla_{X_2}$ again, we get
		\begin{align*}
			\nabla_{X_2}^2(X_1^*) = \nabla_{X_2}(-\Gamma_{2,3}^1 X_3^*) &= -\Gamma_{2,3}^1 \nabla_{X_2}(X_3^*) \\
			&= -\Gamma_{2,3}^1 \left( -(\Gamma_{2,1}^3 X_1^* + \Gamma_{2,2}^3 X_2^* + \Gamma_{2,3}^3 X_3^*) \right) \\
			&= \Gamma_{2,3}^1 \Gamma_{2,1}^3 X_1^*.
		\end{align*}
		The other terms vanish because $\Gamma_{2,2}^3=0$ and $\Gamma_{2,3}^3=0$.
		Similarly, $\nabla_{X_3}(X_1^*) = -(\Gamma_{3,1}^1 X_1^* + \Gamma_{3,2}^1 X_2^* + \Gamma_{3,3}^1 X_3^*) = -\Gamma_{3,2}^1 X_2^*$.
		Applying $\nabla_{X_3}$ again, 
		\begin{align*}
			\nabla_{X_3}^2(X_1^*) = \nabla_{X_3}(-\Gamma_{3,2}^1 X_2^*) &= -\Gamma_{3,2}^1 \nabla_{X_3}(X_2^*) \\
			&= -\Gamma_{3,2}^1 \left( -(\Gamma_{3,1}^2 X_1^* + \Gamma_{3,2}^2 X_2^* + \Gamma_{3,3}^2 X_3^*) \right) \\
			&= \Gamma_{3,2}^1 \Gamma_{3,1}^2 X_1^*.
		\end{align*}
		Combining the terms, we conclude that 
		\begin{equation*}
			C_\nabla^{\su(2)}(X_1^*) = (\Gamma_{2,3}^1 \Gamma_{2,1}^3 + \Gamma_{3,2}^1 \Gamma_{3,1}^2) X_1^*.
		\end{equation*}

It follows from \eqref{eq:christoffel} that
$
\Gamma_{2,3}^1 = (\tfrac{ab}{c}-\tfrac{ac}{b}+\tfrac{bc}{a})$,
$
\Gamma_{2,1}^3 = (-\tfrac{ab}{c}+\tfrac{ac}{b}-\tfrac{bc}{a})$, 
$
\Gamma_{3,2}^1 = (\tfrac{ab}{c}-\tfrac{ac}{b}-\tfrac{bc}{a})$, and 
$
\Gamma_{3,1}^2 = (-\tfrac{ab}{c}+\tfrac{ac}{b}+\tfrac{bc}{a}).
$
		The coefficient is
		\begin{align*}
\Gamma_{2,3}^1 \Gamma_{2,1}^3 + \Gamma_{3,2}^1 \Gamma_{3,1}^2
			&= -(\tfrac{ab}{c}-\tfrac{ac}{b}+\tfrac{bc}{a})^2 - (\tfrac{ab}{c}-\tfrac{ac}{b}-\tfrac{bc}{a})^2 \\
			&= -2\left(\tfrac{a^2b^2}{c^2} + \tfrac{a^2c^2}{b^2} + \tfrac{b^2c^2}{a^2} - 2a^2\right).
		\end{align*}
		This proves the formula for $X_1^*$. The calculations for $X_2^*$ and $X_3^*$ are analogous and yield the stated results by cyclic permutation of indices and parameters. 
	\end{proof}

	We are now in a position to state the explicit form of the Hodge-Laplacian on each isotypic component. By combining the results for the curvature (Lemma \ref{lem:curvature_operator}) and the connection term (Lemma \ref{lem:connection_casimir}), the general formula from Lemma \ref{lem: formula Hodge Laplace general} simplifies significantly. The sum of the fiber-wise operators becomes a simple diagonal matrix, leading to the following expression for the operator:
	\begin{align*}
		\Delta_1^{(\varrho)} = -C_\varrho^{\su(2)}\otimes \Id_{\su(2)_{\C}^*} - \sum_{i=1}^{n}2 d\varrho^*(X_{i})\otimes \nabla_{X_{i}}^{\su(2)} + \underbrace{4\Id_{V_{\varrho}}\otimes\diag\left(\frac{b^{2}c^{2}}{a^{2}},\frac{a^{2}c^{2}}{b^{2}},\frac{a^{2}b^{2}}{c^{2}}\right)}_{= -C_{\nabla}^{\su(2)} + 2q(R)}.
	\end{align*}
	The following proposition provides the detailed matrix entries of this operator.
	
\begin{proposition}\label{prop:Hodge-Laplace_on_k_rep}
The matrix of the Hodge-Laplacian $\Delta_1^{(k)}$ restricted to $\Hom(V_{\varrho_k},\fg_\C^*)\simeq V_{\varrho_k}^*\otimes\fg_\C^*$ is given with respect to the ordered basis $\{e_{(r,p)}=P_{r}\otimes X_{p}^{*}\}_{0\leq r\leq k, 1\leq p\leq 3}$. Its non-zero entries are as follows:
		\begin{align*}
			D_{r,p} &:= (\Delta_1^{(k)})_{(r,p),(r,p)} =
a^2(k-2r)^2 + (b^2+c^2)(k(2r+1)-2r^2)
+ 4\diag\left(\tfrac{b^{2}c^{2}}{a^{2}}, \tfrac{a^{2}c^{2}}{b^{2}}, \tfrac{a^{2}b^{2}}{c^{2}}\right)_{p,p}, \\
			A_r &:= (\Delta_1^{(k)})_{(r,2),(r,3)} = -(\Delta_1^{(k)})_{(r,3),(r,2)} = -2a\mi(k - 2r) \gamma_1, \\
			B_{r,+} &:= (\Delta_1^{(k)})_{(r,1),(r+1,2)} = -(\Delta_1^{(k)})_{(r+1,2),(r,1)} = 2\mi c(r+1) \gamma_3 ,\\
			C_{r,+}&:= {(\Delta_1^{(k)})_{(r,1),(r+1,3)}} = -(\Delta_1^{(k)})_{(r+1,3),(r,1)} = 2b(r+1) \gamma_2, \\
			B_{r,-} &:= {(\Delta_1^{(k)})_{(r,1),(r-1,2)}} = -(\Delta_1^{(k)})_{(r-1,2),(r,1)} = 2\mi c(k-r+1) \gamma_3 ,\\
			C_{r,-} &:= {(\Delta_1^{(k)})_{(r,1),(r-1,3)}} = -(\Delta_1^{(k)})_{(r-1,3),(r,1)} = -2b(k-r+1) \gamma_2, \\
			E_r &:= (\Delta_1^{(k)})_{(r,p),(r-2,p)} = -(C_{\varrho})_{r,r-2} = -(b^2-c^2)(k-r+1)(k-r+2), \\
			F_r &:= (\Delta_1^{(k)})_{(r,p),(r+2,p)} = -(C_{\varrho})_{r,r+2} = -(b^2-c^2)(r+2)(r+1),
		\end{align*}
		where the constants $\gamma_1, \gamma_2, \gamma_3$ are the Christoffel symbols
		\begin{align*}
			\gamma_1 =\Gamma_{1,3}^{2}= -\frac{ab}{c} - \frac{ac}{b} + \frac{bc}{a}, \quad \gamma_2 =\Gamma_{2,3}^{1}= \frac{ab}{c} - \frac{ac}{b} + \frac{bc}{a}, \quad \gamma_3 =\Gamma_{3,2}^{1} = \frac{ab}{c} - \frac{ac}{b} - \frac{bc}{a}.
		\end{align*}
	The explicit representation matrix can be found in \eqref{eq:explicit_Hodge_Laplace}.
		
	\end{proposition}	
	\begin{proof}
		The matrix entries $(\Delta_1^{(k)})_{(r,p),(s,q)}$ of the operator with respect to the basis $\{e_{(r,p)} = P_r \otimes X_p^*\}_{0 \le r \le k, 1 \le p \le 3}$ are determined by the formula
		\begin{equation*}
			(\Delta_1^{(k)})_{(r,p),(s,q)} = -(C_{\varrho_k})_{r,s}\delta_{p,q} - 2\sum_{i=1}^3 (d\varrho_k(X_i))_{r,s} (\nabla_{X_i}^{\su(2)})_{p,q} + 4\delta_{r,s}\delta_{p,q} D_{p,p},
		\end{equation*}
		where $D = \diag(\frac{b^2c^2}{a^2}, \frac{a^2c^2}{b^2}, \frac{a^2b^2}{c^2})$. The non-zero entries of the component matrices for the Casimir operator $C_{\varrho_k}$, the Lie algebra action $d{\varrho_k}(X_i)$, and the connection $\nabla_{X_i}^{\su(2)}$ are known for the indexing $0 \le r, s \le k$ (cf.~\cite{Lauret-SpecSU(2)}, Remark \ref{rem:casimir_correction} and \eqref{eq:christoffel}):
		\begin{itemize}
			\item $(C_{\varrho_k})_{r,r} = -\left( a^2(k-2r)^2 + (b^2+c^2)(k(2r+1)-2r^2) \right)$.
			\item $(C_{\varrho_k})_{r,r+2} = (b^2-c^2)(r+2)(r+1)$.
			\item $(C_{\varrho_k})_{r,r-2} = (b^2-c^2)(k-r+1)(k-r+2)$.
			\item $(d{\varrho_k}(X_1))_{r,r} = a\mi(k-2r)$.
			\item $(d{\varrho_k}(X_2))_{r-1,r} = -br$ and $(d{\varrho_k}(X_2))_{r+1,r} = b(k-r)$.
			\item $(d{\varrho_k}(X_3))_{r-1,r} = -c\mi r$ and $(d{\varrho_k}(X_3))_{r+1,r} = -c\mi(k-r)$.
			\item $(\nabla_{X_1}^{\su(2)})_{2,3}=\Gamma_{1,3}^{2} = \gamma_1 = -(\nabla_{X_1}^{\su(2)})_{3,2}$.
			\item $(\nabla_{X_2}^{\su(2)})_{1,3}=\Gamma_{2,3}^{1} = \gamma_2 = -(\nabla_{X_2}^{\su(2)})_{3,1}$.
			\item $(\nabla_{X_3}^{\su(2)})_{1,2}=\Gamma_{3,2}^{1} = \gamma_3 = -(\nabla_{X_3}^{\su(2)})_{2,1}$.
		\end{itemize}
		All other entries of these component matrices are zero. We now compute the matrix entries of $\Delta_1^{(k)}$ by considering each case for the indices $(s,q)$.
		
		\textit{Case 1: $(s=r, q=p)$.}
		The term with the sum over $i$ vanishes, since $(d{\varrho_k}(X_i))_{r,r}$ is only non-zero for $i=1$, while $(\nabla_{X_1}^{\su(2)})_{p,p}=0$ for all $p$. Thus, the entry is given by
		\begin{equation*}
			(\Delta_1^{(k)})_{(r,p),(r,p)} = -(C_{\varrho_k})_{r,r} + 4D_{p,p} = a^2(k-2r)^2 + (b^2+c^2)(k(2r+1)-2r^2) + 4D_{p,p}.
		\end{equation*}
		
		\textit{Case 2: $(s=r, q\neq p)$.}
		The terms with $C_{\varrho_k}$ and $D$ are zero. The sum over $i$ reduces to the $i=1$ term, as $(d{\varrho_k}(X_2))_{r,r}=(d{\varrho_k}(X_3))_{r,r}=0$. This yields
		\begin{equation*}
			(\Delta_1^{(k)})_{(r,p),(r,q)} = -2 (d{\varrho_k}(X_1))_{r,r} (\nabla_{X_1}^{\su(2)})_{p,q} = -2a\mi(k-2r)(\nabla_{X_1}^{\su(2)})_{p,q}.
		\end{equation*}
		The only non-zero entries occur for $(p,q)=(2,3)$ and $(3,2)$, resulting in
		\begin{align*}
			(\Delta_1^{(k)})_{(r,3),(r,2)} &= 2a\mi(k-2r)\gamma_1	, \\
			(\Delta_1^{(k)})_{(r,2),(r,3)} &= -2a\mi(k-2r)\gamma_1.
		\end{align*}
		
		\textit{Case 3: $(s=r+1)$.}
		The terms with $C_{\varrho_k}$ and $D$ are zero. The non-zero contributions come from the sum over $i=2,3$. The relevant matrix entries for the action $d{\varrho_k}$ are $(d{\varrho_k}(X_i))_{r,r+1}$.
		\begin{align*}
			(\Delta_1^{(k)})_{(r,p),(r+1,q)} &= -2 (d{\varrho_k}(X_2))_{r,r+1} (\nabla_{X_2}^{\su(2)})_{p,q} - 2 (d{\varrho_k}(X_3))_{r,r+1} (\nabla_{X_3}^{\su(2)})_{p,q} \\
			&= -2(-b(r+1))(\nabla_{X_2}^{\su(2)})_{p,q} - 2(-c\mi(r+1))(\nabla_{X_3}^{\su(2)})_{p,q} \\
			&= 2b(r+1)(\nabla_{X_2}^{\su(2)})_{p,q} + 2c\mi(r+1)(\nabla_{X_3}^{\su(2)})_{p,q}.
		\end{align*}
		The only non-zero entries occur when $(p,q)$ is $(1,3)$ or $(1,2)$ (and their transposes):
		\begin{align*}
			(\Delta_1^{(k)})_{(r,1),(r+1,2)} &= 2c\mi(r+1)(\nabla_{X_3}^{\su(2)})_{1,2} = 2c\mi(r+1)\gamma_3, \\
			(\Delta_1^{(k)})_{(r,1),(r+1,3)} &= 2b(r+1)(\nabla_{X_2}^{\su(2)})_{1,3} = 2b(r+1)\gamma_2.
		\end{align*}
		The entries for the transposed indices, e.g., $(\Delta_1^{(k)})_{(r+1,3),(r,1)}$, follow from the anti-symmetry of the $\nabla^{\su(2)}$ matrices.
		
		\textit{Case 4: $(s=r-1)$.}
		Similarly, the relevant matrix entries are $(d{\varrho_k}(X_i))_{r,r-1}$ for $i=2,3$.
		\begin{align*}
			(\Delta_1^{(k)})_{(r,p),(r-1,q)} &= -2 (d{\varrho_k}(X_2))_{r,r-1} (\nabla_{X_2}^{\su(2)})_{p,q} - 2 (d{\varrho_k}(X_3))_{r,r-1} (\nabla_{X_3}^{\su(2)})_{p,q} \\
			&= -2(b(k-r+1))(\nabla_{X_2}^{\su(2)})_{p,q} - 2(-c\mi(k-r+1))(\nabla_{X_3}^{\su(2)})_{p,q} \\
			&= -2b(k-r+1)(\nabla_{X_2}^{\su(2)})_{p,q} + 2c\mi(k-r+1)(\nabla_{X_3}^{\su(2)})_{p,q}.
		\end{align*}
		The non-zero entries are
		\begin{align*}
			(\Delta_1^{(k)})_{(r,1),(r-1,2)} &= 2c\mi(k-r+1)(\nabla_{X_3}^{\su(2)})_{1,2} = 2c\mi(k-r+1)\gamma_3, \\
			(\Delta_1^{(k)})_{(r,1),(r-1,3)} &= -2b(k-r+1)(\nabla_{X_2}^{\su(2)})_{1,3} = -2b(k-r+1)\gamma_2	.
		\end{align*}
		Again, the entries for the transposed indices follow from anti-symmetry.
		
		\textit{Case 5: $(s=r\pm2)$.}
		In these cases, only the Casimir term $-(C_{\varrho_k})_{r,s}\delta_{p,q}$ contributes, which is non-zero only if $p=q$.
		\begin{align*}
			(\Delta_1^{(k)})_{(r,p),(r+2,p)} &= -(C_{\varrho_k})_{r,r+2} =- (b^2-c^2)(r+2)(r+1), \\
			(\Delta_1^{(k)})_{(r,p),(r-2,p)} &= -(C_{\varrho_k})_{r,r-2} =- (b^2-c^2)(k-r+1)(k-r+2).
		\end{align*}
		
		\textit{Case 6: $|s-r| > 2$.}
		All component matrices $(C_{\varrho_k})_{r,s}$, $(d{\varrho_k}(X_i))_{r,s}$, and $\delta_{r,s}$ are zero. Thus, all other matrix entries are zero. This completes the derivation.
	\end{proof}
	\begin{remark}\label{rem:casimir_correction}
		The formulas for the off-diagonal entries of the Casimir operator used here, i.e.\ for the terms $(C_{\varrho_k})_{r,r\pm2}$, include a sign correction to the expressions stated in \cite[Lem.~3.1]{Lauret-SpecSU(2)}. 
		We note that this sign correction does not affect any further results of the cited paper.
	\end{remark}

	\begin{remark}
		The implementation of the Hodge-Laplacian matrices and the symbolic computation of their eigenvalues can be found in the accompanying electronic supplementary material \cite[Sec.~1--3]{HL2025Code}.
		The Hodge-Laplacian $\Delta_1^{(k)}$ is a block-penta-diagonal matrix, where the $3 \times 3$ blocks are indexed by $r=0, \dots, k$:
		\begin{equation}\label{eq:block_matrix_structure}
			\Delta_1^{(k)} = 
			\begin{pmatrix}
				\mathbf{M}_0 & \mathbf{K}_0^+ & \mathbf{F}_0 & & & \mathbf{0} \\
				\mathbf{K}_1^- & \mathbf{M}_1 & \mathbf{K}_1^+ & \mathbf{F}_1 & & \\
				\mathbf{E}_2 & \mathbf{K}_2^- & \mathbf{M}_2 & \mathbf{K}_2^+ & \ddots & \\
				& \ddots & \ddots & \ddots & \ddots & \\
				\mathbf{0} & & & \mathbf{E}_k & \mathbf{K}_k^- & \mathbf{M}_k
			\end{pmatrix},
		\end{equation}
and the blocks are given by
		\begin{align*}
			\mathbf{M}_r &= \begin{pmatrix} D_{r,1} & 0 & 0 \\ 0 & D_{r,2} & A_r \\ 0 & -A_r & D_{r,3} \end{pmatrix}, \quad
			\mathbf{K}_r^{\pm} = \begin{pmatrix} 0 & B_{r,\pm} & C_{r,\pm} \\ -B_{r,\pm} & 0 & 0 \\ -C_{r,\pm} & 0 & 0 \end{pmatrix}, \\
			\mathbf{E}_r &= \diag(E_r, E_r, E_r), \quad\quad\quad\:\: \mathbf{F}_r = \diag(F_r, F_r, F_r).
		\end{align*}
		Here, $\mathbf{K}_r^+$ is the block in the upper diagonal and $\mathbf{K}_r^-$ is in the lower diagonal. The explicit formulas for all scalar entries are given in Proposition~\ref{prop:Hodge-Laplace_on_k_rep}.
		More explicitly, the matrix takes the form
		\begin{equation}\label{eq:explicit_Hodge_Laplace}
				\begin{pmatrix}
					D_{0,1} & 0 & 0 & 0 & B_{0,+} & C_{0,+} & F_0 & 0 & 0 & \dots \\
					0 & D_{0,2} & A_0 & -B_{0,+} & 0 & 0 & 0 & F_0 & 0 & \dots \\
					0 & -A_0 & D_{0,3} & -C_{0,+} & 0 & 0 & 0 & 0 & F_0 & \dots \\
					0 & B_{1,-} & C_{1,-} & D_{1,1} & 0 & 0 & 0 & B_{1,+} & C_{1,+} & \dots \\
					-B_{1,-} & 0 & 0 & 0 & D_{1,2} & A_1 & -B_{1,+} & 0 & 0 & \dots \\
					-C_{1,-} & 0 & 0 & 0 & -A_1 & D_{1,3} & -C_{1,+} & 0 & 0 & \dots \\
					E_2 & 0 & 0 & 0 & B_{2,-} & C_{2,-} & D_{2,1} & 0 & 0 & \dots \\
					0 & E_2 & 0 & -B_{2,-} & 0 & 0 & 0 & D_{2,2} & A_2 & \dots \\
					0 & 0 & E_2 & -C_{2,-} & 0 & 0 & 0 & -A_2 & D_{2,3} & \dots \\
					\vdots & \vdots & \vdots & \vdots & \vdots & \vdots & \vdots & \vdots & \vdots & \ddots
				\end{pmatrix}.
		\end{equation}

	\end{remark}
	\subsection{The First Eigenvalue Conjecture}\label{subsec:firsteigenvalue}
	
	The smallest non-zero eigenvalue $\lambda_1(\Delta_1)$ of the Hodge-Laplacian is of particular interest. Our analysis of the operator's matrix structure suggests that this eigenvalue occurs for a small representation index $k$. One of those verifications can be found in \cite[Sec.~4]{HL2025Code}. We therefore compute the spectra for the first two isotypic components, $k=0$ and $k=1$, explicitly. The implementation of the Hodge-Laplacian matrices and the symbolic computation of their eigenvalues can be found in the accompanying Jupyter Notebook \cite[Sec.~1--3]{HL2025Code}.
	
	For $k=0$, the operator $\Delta_1^{(0)}$ is a diagonal $3 \times 3$ matrix, yielding the eigenvalues
	\begin{equation}\label{eq:spec_k0}
		\frac{4b^2c^2}{a^2}, \quad \frac{4a^2c^2}{b^2}, \quad \frac{4a^2b^2}{c^2}.
	\end{equation}
	
	For $k=1$, the $6 \times 6$ matrix $\Delta_1^{(1)}$ is given by
	\begin{align*}
		\resizebox{\textwidth}{!}{%
			$
			\setlength{\arraycolsep}{2.5pt} 
			\renewcommand{\arraystretch}{1.5} 
			\begin{pmatrix}
			a^{2} + b^{2} + \frac{4 \, b^{2} c^{2}}{a^{2}} + c^{2} & 0 & 0 & 0 & 2 i \, b^{2} - \frac{2 i \, b^{2} c^{2}}{a^{2}} - 2 i \, c^{2} & 2 \, b^{2} + \frac{2 \, b^{2} c^{2}}{a^{2}} - 2 \, c^{2} \\\\[6mm]
			0 & a^{2} + b^{2} + \frac{4 \, a^{2} c^{2}}{b^{2}} + c^{2} & 2 i \, a^{2} + \frac{2 i \, a^{2} c^{2}}{b^{2}} - 2 i \, c^{2} & -2 i \, a^{2} + \frac{2 i \, a^{2} c^{2}}{b^{2}} + 2 i \, c^{2} & 0 & 0 \\\\[6mm]
			0 & -2 i \, a^{2} + 2 i \, b^{2} - \frac{2 i \, a^{2} b^{2}}{c^{2}} & a^{2} + b^{2} + \frac{4 \, a^{2} b^{2}}{c^{2}} + c^{2} & 2 \, a^{2} - 2 \, b^{2} - \frac{2 \, a^{2} b^{2}}{c^{2}} & 0 & 0 \\\\[6mm]
			0 & 2 i \, b^{2} - \frac{2 i \, b^{2} c^{2}}{a^{2}} - 2 i \, c^{2} & -2 \, b^{2} - \frac{2 \, b^{2} c^{2}}{a^{2}} + 2 \, c^{2} & a^{2} + b^{2} + \frac{4 \, b^{2} c^{2}}{a^{2}} + c^{2} & 0 & 0 \\\\[6mm]
			-2 i \, a^{2} + \frac{2 i \, a^{2} c^{2}}{b^{2}} + 2 i \, c^{2} & 0 & 0 & 0 & a^{2} + b^{2} + \frac{4 \, a^{2} c^{2}}{b^{2}} + c^{2} & -2 i \, a^{2} - \frac{2 i \, a^{2} c^{2}}{b^{2}} + 2 i \, c^{2} \\\\[6mm]
			-2 \, a^{2} + 2 \, b^{2} + \frac{2 \, a^{2} b^{2}}{c^{2}} & 0 & 0 & 0 & 2 i \, a^{2} - 2 i \, b^{2} + \frac{2 i \, a^{2} b^{2}}{c^{2}} & a^{2} + b^{2} + \frac{4 \, a^{2} b^{2}}{c^{2}} + c^{2}\\[6mm]
		\end{pmatrix}$
	}
	\end{align*}
 One can check that it has three distinct eigenvalues, each with multiplicity 2. One eigenvalue corresponds to the exact 1-forms
	\begin{equation*}
		\lambda_{\text{exact}} := a^2+b^2+c^2.
	\end{equation*}
	The other two eigenvalues, corresponding to coclosed eigenforms, are given by
	\begin{align*}
		\lambda_{\text{coclosed}}^{\pm} :=&\frac{
			2\, a^{4} b^{4}
			+\bigl(2\, a^{4} + a^{2} b^{2} + 2\, b^{4}\bigr) c^{4}
			+\bigl(a^{4} b^{2} + a^{2} b^{4}\bigr) c^{2}}{a^{2} b^{2} c^{2}}\\
		&\pm\frac{ 2\, \sqrt{a^{4} b^{4} + \bigl(a^{4} - a^{2} b^{2} + b^{4}\bigr) c^{4} - \bigl(a^{4} b^{2} + a^{2} b^{4}\bigr) c^{2}}\,
			\bigl(a^{2} b^{2} + (a^{2}+b^{2}) c^{2}\bigr)}
		{a^{2} b^{2} c^{2}}.
	\end{align*}
	The following theorem establishes the crucial relationship between these eigenvalues.
	
	\begin{proposition}\label{prop:k1_eigenvalue_comparison}
		For $a,b,c>0$, the first eigenvalue of the union of the spectra of $\Delta_1^{(0)}$ and $\Delta_1^{(1)}$ is given by
		\[
		\lambda_1=\min\!\left(\frac{4a^2b^2}{c^2},\; \frac{4a^2c^2}{b^2},\;\frac{4b^2c^2}{a^2},\;a^2+b^2+c^2\right).
		\]
	\end{proposition}

\begin{proof} 
It suffices to prove that $\lambda_1 \leq \lambda_{\text{coclosed}}^{-}$. 
It is a simple matter to check that 
\begin{align}\label{eq:lambda_coclosed}
\lambda_{\text{coclosed}}^{-}& 
= \frac{(M-S)^2}{a^2b^2c^2},
&
\text{where }
\begin{cases}
M=a^2b^2+(a^2+b^2)c^2,\\
S=\sqrt{a^4b^4+(a^4-a^2b^2+b^4)c^4-a^2b^2(a^2+b^2)c^2}.
\end{cases}
\end{align}

We will require the elementary estimate for $S$:
\begin{equation}\label{eq:star}
\text{If } 3c^2\ge a^2+b^2, \text{ then } 
			S\le (a^2+b^2)c^2 - a^2b^2. 
\end{equation}
Indeed, $\bigl((a^2+b^2)c^2-a^2b^2\bigr)^2 - S^2=a^2b^2c^2\,(3c^2-a^2-b^2)\ge 0.$
		
We now proceed by a case analysis depending on whether the minimum for $\lambda_1$ is attained. 
We first assume that $\lambda_1=a^2+b^2+c^2$.
It follows that
$a^2+b^2+c^2\leq \frac{4a^2b^2}{c^2}$,
$a^2+b^2+c^2\leq \frac{4a^2c^2}{b^2}$, and
$a^2+b^2+c^2\leq \frac{4b^2c^2}{a^2}$.
From the latter two we obtain
$
(a^2+b^2+c^2)^2\leq 
16c^4
$, thus $3c^2\geq a^2+b^2$.
Now, \eqref{eq:star} forces $S\leq (a^2+b^2)c^2-a^2b^2$, 
which gives $M-S\ge 2a^2b^2$.
It follows from \eqref{eq:lambda_coclosed} that 
\[
\lambda_{\text{coclosed}}^{-}
\geq 
\frac{4a^2b^2}{c^2}
\geq a^2+b^2+c^2=\lambda_1
.
\]

We now suppose that $\lambda_1=\min\{ \frac{4a^2b^2}{c^2},\; \frac{4a^2c^2}{b^2},\; \frac{4b^2c^2}{a^2}\}$.
Since $\lambda_1$ and $\lambda_{\text{coclosed}}^{-}$ are invariant under permutations of $a,b,c$, we can assume without losing generality that $\lambda_1 =\frac{4a^2b^2}{c^2}$.
Since $\frac{4a^2b^2}{c^2}\leq \frac{4a^2c^2}{b^2}$ and $\frac{4a^2b^2}{c^2}\leq \frac{4b^2c^2}{a^2}$,  it follows that $c^2\geq \max\{a^2,b^2\}> \tfrac{a^2+b^2}{3}$. 
Using \eqref{eq:star}, we obtain again that 
$M-S\ge 2a^2b^2$, and \eqref{eq:lambda_coclosed} gives
$
\lambda_{\text{coclosed}}^{-}
\geq 
\frac{4a^2b^2}{c^2}
=\lambda_1.
$
This completes the proof. 
\end{proof}
For higher representation weights ($k \ge 2$), determining the exact minimum from the block matrices $\Delta_1^{(k)}$ analytically is difficult. Since there are infinitely many representations ($k \to \infty$), we use a two-tiered computer-assisted approach in the electronic supplement \cite[Sec.~4]{HL2025Code} to locate the global minimum.

First, the matrix entries grow asymptotically as $\mathcal{O}(k^2)$. For a large proportion of metrics satisfying the condition $a^2 > |b^2-c^2|$ (and cyclic permutations), the diagonal entries dominate this growth. We apply the Gershgorin Circle Theorem dynamically: beyond a certain threshold $k_0$, the theorem provides a strictly increasing lower bound for all eigenvalues. This guarantees that the global minimum over all infinitely many representations must be located within the finite set of representations with $k < k_0$. An algorithmic check of this finite set rigorously proves that the global minimum is attained at $k \in \{0,1\}$ for these metrics.

Second, for the remaining metrics in this subset, and for highly asymmetric metrics, the off-diagonal entries dominate. The Gershgorin circles overlap with the negative real axis and do not provide a useful lower bound. To address these cases, we use a Monte Carlo stress test. We compute the exact eigenvalues for a sample of 10,000 random metrics up to weight $k=10$. 


The numerical results consistently show that the minimum is always attained at $k \in \{0,1\}$.
This led the authors to formulate, during the research process, the following conjecture.
(See Section~5 for an AI-assisted proof.)

\begin{conjecture}\label{conj:first_eigenvalue}
The smallest eigenvalue of the Hodge-Laplacian on 1-forms associated to $(\SU(2),g_{\abc})$ is given by
\begin{equation*}
\lambda_1(\Delta_1) 
= \min\left( \frac{4b^2c^2}{a^2}, \frac{4a^2c^2}{b^2}, \frac{4a^2b^2}{c^2}, a^2+b^2+c^2 \right).
\end{equation*}
Analogously, for $(\SO(3),g_{\abc})$, 
\begin{equation*}
\lambda_1(\Delta_1) 
= \min\left( \frac{4b^2c^2}{a^2}, \frac{4a^2c^2}{b^2}, \frac{4a^2b^2}{c^2} \right).
\end{equation*}
\end{conjecture}
	
We present a proof of this conjecture in Section \ref{sec:proof_conjecture}.

	\section{The case of Berger spheres}\label{sec:Berger spheres}

In this section we focus in ($3$-dimensional) Berger spheres. 
Any of them is isometric to $(\SU(2),g_{\abc})$ with $b=c$. 
We assume throughout the section that $a,b$ are positive real numbers, $\Delta_1$ denotes the Hodge-Laplace operator acting on $1$-forms associated to $(\SU(2),g_{(a,b,b)})$, $G=\SU(2)$ and $\fg=\su(2)$.

From Subsection~\ref{subsec:firsteigenvalue}, taking $b=c$, it follows immediately that the spectrum of $\Delta_1^{(0)}$ and $\Delta_1^{(1)}$ are given by 
\begin{equation}
\begin{aligned}
\Spec(\Delta_1^{(0)})&
=\{4a^2, 4\tfrac{b^4}{a^2}, 4\tfrac{b^4}{a^2}\}
,\\
\Spec(\Delta_1^{(1)})&
=\{
		9a^2, 
		9a^2, 
		a^2+4b^2+4\tfrac{b^4}{a^2},
		a^2+4b^2+4\tfrac{b^4}{a^2},
		a^2+2b^2,
		a^2+2b^2
		\}.
\end{aligned}
\end{equation}

The next goal is to describe the spectrum of $\Delta_1^{(k)}$ for any $k\geq2$. 
We recall that $\{P_0,\dots,P_k\}$ is the basis of $V_{\varrho_k}$ given by $P_l(z,w)=z^lw^{k-l}$ for any $0\leq l\leq k$, $\{X_1,X_2,X_3\}$ is the orthonormal basis of $\su(2)$ with respect to $\innerdots_{\abc}$ defined in Subsection~\ref{subsec:SU(2)}, and $\{X_1^*,X_2^*,X_3^*\}$ is its dual basis. 
As in Proposition~\ref{prop:Hodge-Laplace_on_k_rep}, we consider the basis $\{P_r\otimes X_p^*: 0\leq r\leq k,\, 1\leq p\leq 3\}$ of $V_{\varrho_k}\otimes \fg_\C^*$.

We fix an integer $k\geq2$. 
For any $j=0,\dots,k$, we set 
\begin{equation}
v_{k,j}
=  j\tfrac{b}{a}\mi\, P_{j-1}\otimes (X_2+\mi X_3)
+ (k-2j)\, P_j\otimes X_1
- (k-j) \tfrac{b}{a}\mi\, P_{j+1}\otimes (X_2-\mi X_3)
.
\end{equation}
Furthermore, we set
\begin{equation}
\begin{aligned}	
w_{k,-1}&
=P_0\otimes(X_2-\mi X_3),
&\qquad 
w_{k,0}&
= 2\tfrac{b}{a}\, P_0\otimes X_1 
+ \mi k P_{1}\otimes (X_2-\mi X_3)
,
\\
w_{k,k+1}&
=P_k\otimes(X_2+\mi X_3)
,
&
w_{k,k}&
= \mi k P_{k-1}\otimes (X_2+\mi X_3)
+2\tfrac{b}{a}\, P_k\otimes X_1
,
\end{aligned}
\end{equation}
and, for any $1\leq j\leq k-1$, 
\begin{equation}
		w_{k,j}^\pm= \mi \alpha_j^\pm P_{j-1}\otimes (X_2+\mi X_3) + \beta_j^\pm P_j\otimes X_1 + \mi \gamma_j^\pm P_{j+1}\otimes (X_2-\mi X_3)
		,
\end{equation}
where
\begin{equation}
		\begin{aligned}
			\alpha_j^\pm&
			= 
			\frac{1}{4(k-j+1)}\frac{1}{b^2} 
			\left(\begin{array}{l}
				2(k-2j)^2a^2
				+\big(2k+4j(k-j)\big)b^2
				\\
				\pm (-2)(k-2j) \sqrt{a^4(k-2j)^2 +a^2b^2\big(2k+4j(k-j)\big)+b^4}
			\end{array}\right)
			,\\
			\beta_j^\pm&
			=
			\frac{-a^2(k-2j) + b^2 \pm \sqrt{a^4(k-2j)^2 +a^2b^2\big(2k+4j(k-j)\big)+b^4} }{ab}
			,\\
			\gamma_j^\pm&
			=k-j
			.
		\end{aligned}
\end{equation}

\begin{theorem}\label{thm4:Spec(Berger)}
Fix $k\geq2$. 
A basis of eigenvectors of $\Delta_1^{(k)}$ is given by 
\begin{equation}\label{eq4:basis-eigenvectors}
\{v_{k,j}: 0\leq j\leq k\} \cup
\{w_{k,-1},w_{k,0},w_{k},w_{k,k+1}\}\cup
\{w_{k,j}^\pm: 1\leq j\leq k-1\}
.
\end{equation}
The associated eigenvalues are given as follows:
\begin{align}
\Delta_1^{(k)}\cdot v_{k,j}&=\nu_{k,j}\, v_{k,j},&
\Delta_1^{(k)}\cdot w_{k,j}&=\mu_{k,j}\, w_{k,j},&
\Delta_1^{(k)}\cdot w_{k,j}^\pm&=\mu_{k,j}^\pm\, w_{k,j}^\pm,
\end{align} 
for any $j$ indicated as in \eqref{eq4:basis-eigenvectors}, where
\begin{equation}
\begin{aligned}
\nu_{k,j}&={a^2(k-2j)^2 + b^2((4j+2)k-4j^{2})},
\\
\mu_{k,-1}=\mu_{k,k+1}&=(k+2)^2a^2,
\\
\mu_{k, 0}=\mu_{k,k  }&=k^2a^2+4kb^2 +4\tfrac{b^4}{a^2},
\\
\mu_{k,j}^\pm&
=	\nu_{k,j}+2 \tfrac{b^4}{a^2} \pm 2 \tfrac{b^2}{a^2} \sqrt{a^2\cdot\nu_{k,j} +b^4}.
\end{aligned}
\end{equation}
\end{theorem}

\begin{proof}
Let $(r,p)$ and $(s,q)$ be indices of the basis of $V_{\varrho_k}\otimes\fg_\C^*$, so $0\leq r,s\leq k$ and $1\leq p,q\leq 3$. 
Taking into account the assumption $b=c$, 
Proposition~\ref{prop:Hodge-Laplace_on_k_rep} ensures that
\begin{align*}
(\Delta_1^{(k)})_{(r,p),(s,q)}
		=
		-(C_{\varrho_k})_{(r,s)} \delta_{(p,q)} 
		+ 4\delta_{(r,s)}D_{p,q}
		- 2\sum_{i=1}^3 (d{\varrho_k}(X_i))_{(r,s)} (\nabla_{X_i}^{\fg})_{(p,q)}
		,
\end{align*}
where 
\begin{equation}\label{eq4:coeff-b=c}
\begin{array}{r@{\;}l@{\qquad}r@{\;}l}
D&=
\diag(\frac{b^4}{a^2}, a^2, a^2),
&
-(C_{\varrho_k})_{r,r} =\nu_{k,r}&=
{a^2(k-2r)^2 + b^2((4r+2)k-4r^{2})},
\\
(d\varrho_k(X_1))_{r,r} &= a \mi (k - 2r).
\\
(d\varrho_k(X_2))_{r,r-1} &= b(k-r+1), &
(d\varrho_k(X_2))_{r,r+1} &= -b(r+1).
\\
(d\varrho_k(X_3))_{r,r-1} &= -b\mi (k-r+1), &
(d\varrho_k(X_3))_{r,r+1} &=- b\mi (r+1).
\\
(\nabla_{X_1}^{\fg})_{3,2} &
=(2a-\tfrac{b^2}{a}),
&
(\nabla_{X_1}^{\fg})_{2,3} &
=-(2a-\tfrac{b^2}{a}),
\\
(\nabla_{X_2}^{\fg})_{3,1} &
=-\tfrac{b^2}{a}
&
(\nabla_{X_2}^{\fg})_{1,3} &
=\tfrac{b^2}{a}, 
\\ 
(\nabla_{X_3}^{\fg})_{2,1} &
=\tfrac{b^2}{a} ,
&
(\nabla_{X_3}^{\fg})_{1,2} &
=-\tfrac{b^2}{a},
\end{array}
\end{equation}
and all other entries of these component matrices are zero. 
Consequently, 
\begin{equation}\label{eq4:matrix-b=c}
\begin{aligned}
(\Delta_1^{(k)})_{(r,p),(s,q)} &
			=
			\big( \nu_{k,r} + 4D_{p,p} \big) \cdot \delta_{r,s} \delta_{p,q} 
			-2\mi (k-2r) (2a^2-b^2) \cdot \delta_{r,s}
			\big(\delta_{p,3}\delta_{q,2}-\delta_{p,2}\delta_{q,3}\big)
			\\ & \quad 
			+ \tfrac{b^3}{a}\big(
			-2 (k-r+1) \delta_{s,r-1}  
			+2 (r+1)   \delta_{s,r+1} 
			\big) \cdot \big(\delta_{p,1}\delta_{q,3}-\delta_{p,3}\delta_{q,1}\big)
			\\ & \quad 
			+ \tfrac{b^3}{a} \big(
			+2\mi (k-r+1) \delta_{s,r-1} 
			+2\mi (r+1)   \delta_{s,r+1} 
			\big)  \cdot \big(\delta_{p,2}\delta_{q,1}-\delta_{p,1}\delta_{q,2}\big)
			.
\end{aligned}
\end{equation}

We are now ready to prove the assertions claimed. 
For any $0\leq j\leq k$, we begin computing  
\begin{equation*}
\begin{aligned}
(\Delta_1^{(k)} \cdot v_{k,j})_{(r,p)}&
= \sum_{(s,q)} (\Delta_1^{(k)})_{(r,p),(s,q)} (v_{k,j})_{(s,q)}
\\ & 
= \mi j\tfrac{b}{a} (\Delta_1^{(k)})_{(r,p),(j-1,2)} 
-  j\tfrac{b}{a} (\Delta_1^{(k)})_{(r,p),(j-1,3)} 
+ (k-2j) (\Delta_1^{(k)})_{(r,p),(j,1)}
\\ &\quad 
- \mi (k-j)\tfrac{b}{a} (\Delta_1^{(k)})_{(r,p),(j+1,2)}
- (k-j)\tfrac{b}{a} (\Delta_1^{(k)})_{(r,p),(j+1,3)}
.
\end{aligned}
\end{equation*} 
Now, \eqref{eq4:matrix-b=c} yields
\begin{equation*}
\begin{aligned}
(\Delta_1^{(k)} \cdot v_{k,j})_{(r,p)}&
				= \mi j\tfrac{b}{a} 
				\left(
				-(C_\varrho)_{(r,j-1)}\cdot \delta_{(p,2)} 
				+ 4\delta_{(r,j-1)}\cdot D_{p,2}
				- \sum_{i=1}^3 2  (d\varrho(X_i))_{(r,j-1)}\cdot (\nabla_{X_i}^{\fg})_{(p,2)}
				\right)
				\\ & \quad 
				-  j\tfrac{b}{a} 
				\left(
				-(C_\varrho)_{(r,j-1)}\cdot \delta_{(p,3)} 
				+ 4\delta_{(r,j-1)}\cdot D_{p,3}
				- \sum_{i=1}^3 2  (d\varrho(X_i))_{(r,j-1)}\cdot (\nabla_{X_i}^{\fg})_{(p,3)}
				\right)
				\\ & \quad 
				+(k-2j)\left(
				-(C_\varrho)_{(r,j)}\cdot \delta_{(p,1)} 
				+ 4\delta_{(r,j)}\cdot D_{p,1}
				- \sum_{i=1}^3 2  (d\varrho(X_i))_{(r,j)}\cdot (\nabla_{X_i}^{\fg})_{(p,1)}
				\right)
				\\ & \quad 
				- \mi (k-j)\tfrac{b}{a} 
				\left(
				-(C_\varrho)_{(r,j+1)}\cdot \delta_{(p,2)} 
				+ 4\delta_{(r,j+1)}\cdot D_{p,2}
				- \sum_{i=1}^3 2  (d\varrho(X_i))_{(r,j+1)}\cdot (\nabla_{X_i}^{\fg})_{(p,2)}
				\right)
				\\ & \quad 
				-(k-j)\tfrac{b}{a} 
				\left(
				-(C_\varrho)_{(r,j+1)}\cdot \delta_{(p,3)} 
				+ 4\delta_{(r,j+1)}\cdot D_{p,3}
				- \sum_{i=1}^3 2  (d\varrho(X_i))_{(r,j+1)}\cdot (\nabla_{X_i}^{\fg})_{(p,3)}
				\right)
.
\end{aligned} 
\end{equation*}
After substituting \eqref{eq4:coeff-b=c} above, tedious but straightforward calculations give
\begin{equation*}
\begin{aligned}
(\Delta_1^{(k)} \cdot v_{k,j})_{(r,p)}&
				= 
				\delta_{r,j-1}\delta_{p,2}\, \mi j\tfrac{b}{a} \nu_{k,j}
				+\delta_{r,j-1}\delta_{p,3}\, j\tfrac{b}{a} (-1)\nu_{k,j}
				+\delta_{r,j}\delta_{p,1} \,(k-2j) \nu_{k,j} 
				\\ &\quad
				+\delta_{r,j+1}\delta_{p,2}\, \mi(k-j)\tfrac{b}{a}(-1)\nu_{k,j}
				+\delta_{r,j+1}\delta_{p,3}\, (k-j)\tfrac{b}{a}(-1)\nu_{k,j} 
.
\end{aligned} 
\end{equation*}
We thus conclude that 
\begin{equation*}
\begin{aligned}
\Delta_1^{(k)}\cdot v_{k,j} &
				=\sum_{(r,p)} (\Delta_1^{(k)}\cdot v_{k,j})_{(r,p)} \; P_r\otimes X_p
				= \nu_{k,j}\,v_{k,j}. 
			\end{aligned}
		\end{equation*}
The proofs of $\Delta_1^{(k)}\cdot w_{k,j}=\mu_{k,j}\, w_{k,j}$ for $j=-1,0,k,k+1$ are very similar, so they are left to the reader. 

We now move to prove that $w_{k,j}^\pm$ is an eigenfunctions for $\Delta_1^{(k)}$ for $1\leq j\leq k-1$. 
Proceeding as above, one gets
\begin{equation}
(\Delta_1^{(k)}\cdot u_j)_{(r,p)} 
=
\mi A_j^\pm \, (\delta_{r,j-1} \delta_{p,2} 
+\mi \delta_{r,j-1} \delta_{p,3} )
+\delta_{r,j} \delta_{p,1} B_j^\pm
+\mi C_j^\pm\,  (\delta_{r,j+1} \delta_{p,2} 
-\mi \delta_{r,j+1} \delta_{p,3})
,
\end{equation}
where
\begin{equation}
\begin{aligned}
A_j^\pm&=\alpha_j^\pm \big(
(k-2j)^2 a^2
+ 4j(k-j+1)b^2 
\big)
+ 2j \beta_j^\pm\tfrac{b^3}{a}
, \\ 
B_j^\pm&=
+4 \tfrac{b^3}{a} \big(
\alpha_j^\pm (k-j+1) 
+\gamma_j^\pm (j+1)  
\big)
+\beta_j^\pm \big( \nu_{k,j} + 4\tfrac{b^4}{a^2} \big)
, \\
C_j^\pm&=\gamma_j^\pm\big(
(k-2j)^2 a^2
+ 4(j+1)(k-j)b^2 
\big)
+ 2(k-j) \beta_j^\pm \tfrac{b^3}{a}
.
\end{aligned}
\end{equation}
Hence
\begin{equation}
\begin{aligned}
\Delta_1^{(k)}\cdot u_j &
= \sum_{(r,p)} (\Delta\cdot u_j)_{(r,p)} P_r\otimes X_p
\\ & 
=A_j^\pm\mi P_{j-1}\otimes (X_2+\mi X_3)
+B_j^\pm\mi P_{j}\otimes X_1
+C_j^\pm\mi P_{j+1}\otimes (X_2-\mi X_3)
.
\end{aligned}
\end{equation}
It is a simple matter to check that 
$A_j^\pm=\alpha_j^\pm\mu_{k,j}^\pm$,
$B_j^\pm=\beta_j^\pm\mu_{k,j}^\pm $,
$C_j^\pm=\gamma_j^\pm\mu_{k,j}^\pm$,
and the identity $\Delta_1^{(k)}\cdot w_{k,j}^\pm = \mu_{k,j}^\pm \, w_{k,j}^\pm$ follows. 
\end{proof}
We are now in position to state the full spectrum of $\Delta_1$ associated to $(\SU(2),g_{(a,b,b)})$ and $(\SO(3), g_{(a,b,b)})$ for any $a,b>0$. 

\begin{corollary}\label{cor:full spectrum Berger}
The spectrum of the Hodge-Laplace operator $\Delta_1$ acting on $1$-forms associated to the Berger sphere $(\SU(2),g_{(a,b,b)})$ is given by the union over all $k\geq 0$ of the following $\Delta_{1}^{(k)}$-spectra:
\begin{equation}
\begin{aligned}
\Spec(\Delta_1^{(0)})&
=\{4a^2, 4a^2, 4\tfrac{b^4}{a^2}\}
\\
\Spec(\Delta_{1}^{(k)})&
=
	\Big\{
	\underbrace{\nu_{k,j},\dots,\nu_{k,j}}_{\text{$(k+1)$-times}}
	:0\leq j\leq k
	\Big\}
	\cup \;
	\Big\{
		\underbrace{\mu_{k,j},\dots,\mu_{k,j}}_{\text{$(k+1)$-times}}
	:j=-1,0,k,k+1
	\Big\}
\\ & \quad
	\cup \;
	\Big\{
	\underbrace{\mu_{k,j}^+,\mu_{k,j}^- ,\dots, \mu_{k,j}^+,\mu_{k,j}^-}_{\text{$(k+1)$-times}}
	:1\leq j\leq k-1
	\Big\}
\qquad\text{for any $k\geq0$.}
\end{aligned}\end{equation}

For the projective space $(\SO(3), g_{(a,b,b)})$, the spectrum is the subset consisting only of those eigenvalues corresponding to even weights $k \in 2\mathbb{N}_0$, that is, 
\begin{equation}
	\bigcup_{k \in 2\mathbb{N}_0} \Spec(\Delta_{1}^{(k)}).
\end{equation}
\end{corollary}

\begin{remark}\label{rem: interpretation Berger}
The spectrum of the Laplace-Beltrami operator $\Delta_0$ associated to the Berger $3$-sphere $(\SU(2),g_{(a,b,b)})$ was determined by Tanno~\cite[Lem.~4.1]{Tanno79} (see also \cite[Prop.~3.9]{Lauret-SpecSU(2)}):
\begin{equation}
\bigcup_{k\geq0}
	\Big\{
	\underbrace{\nu_{k,j},\dots,\nu_{k,j}}_{\text{$(k+1)$-times}}
	:0\leq j\leq k
	\Big\}
.
\end{equation}

Therefore, $\Spec(\Delta_0)\smallsetminus\{0\} \subset \Spec(\Delta_1)$, which is consistent with the Hodge decomposition of $1$-forms among harmonics, closed and coclosed  $1$-forms. 
Indeed, if $f$ is a non-constant eigenfunction of $\Delta_0$, then $df$ is a closed eigen-$1$-form of $\Delta_1$ associated to the same eigenvalue.\\
The eigenvalues $\mu_{k,-1}=\mu_{k,k+1} = (k+2)^2 a^2$ correspond to higher eigenvalues of the Laplacian on the fiber $S^1$.
Furthermore, the mixed eigenvalues exhibit a  geometric structure related to the geometry of the non-trivial bundle. Let $\omega = E_{1}^{*} \otimes E_1 \in \Omega^1(S^3, \mathfrak{su}(2))$ be the connection $1$-form of the Hopf fibration and $\Omega = d\omega$ its curvature form. We define the curvature parameter $\kappa$ via the squared norm of $\Omega$:
\[ \kappa := \frac{1}{4} \|\Omega\|^2 = \frac{1}{4} \|-2b^2 X_2^* \wedge X_3^* \otimes E_1\|^2 = \frac{b^4}{a^2}. \]
This coincides with the squared norm of the O'Neill tensor $\|A\|^2$. In terms of $\kappa$, the mixed eigenvalues take the compact form
\begin{equation*}
	\begin{aligned}
		\mu_{k,0}=\mu_{k,k} &= \big(ka+2\sqrt{\kappa}\big)^2, \\
		\mu_{k,j}^{\pm} &= \big(\sqrt{\nu_{k,j}+\kappa}\ \pm\ \sqrt{\kappa}\big)^2.
	\end{aligned}
\end{equation*}
\end{remark}
From the explicit description of the full spectrum in Corollary \ref{cor:full spectrum Berger}, we can directly determine the smallest non-zero eigenvalue. This provides an independent verification of the First Eigenvalue formula (Theorem \ref{thm:first eigenvalue}) for this specific sub-family of metrics.
	\begin{corollary}
The smallest eigenvalue of the Hodge-Laplacian on 1-forms for the Berger sphere $(\SU(2), g_{(a,b,b)})$ is given by
		\begin{equation}
\min \left\{\nu_{1,0}, \mu_{0,0}, \mu_{0,\pm 1} \right\}
= \min \left\{a^2+2b^2, \frac{4b^4}{a^2}, 4a^2 \right\}
.
\end{equation}
For the projective space $(\SO(3), g_{(a,b,b)})$, the smallest eigenvalue is given by
\begin{equation}
\min\left\{ \frac{4b^4}{a^2}, 4a^2 \right\}.
\end{equation}
		Note that the eigenvalues $\nu_{1,j}$ appear with total multiplicity $2(1+1)=4$, while $\mu_{0,0}$ has multiplicity $1$ and the pair $\mu_{0,\pm 1}$ has total multiplicity $2$.
	\end{corollary}

\section{Proof of the First Eigenvalue Conjecture}
\label{sec:proof_conjecture}

In this section, we provide a rigorous proof of the First Eigenvalue Conjecture (Conjecture \ref{conj:first_eigenvalue}). As outlined in the introduction, the core analytical strategy for this proof---transforming the Curl operator to the round metric to bound the coexact spectrum---was autonomously discovered by ChatGPT 5.4 Pro. It took the model 100 minutes of non-stop computational reasoning to derive this mathematical result. We note that this conclusion is based on a computationally intensive, highly restricted tier (USD 200/month, limited to 15 queries), which exceeds the capabilities of standard, freely available AI tools by far. 
In what follows, we detail the complete mathematical argument, rewritten and carefully embedded in the relevant literature.

\subsection{The curl operator and the Hodge-Laplacian}

On a Riemannian manifold $(M, g)$ of dimension $n$, the composition $*_g d$ defines an endomorphism on $p$-forms if and only if $n = 2p + 1$. In our specific setting of $3$-manifolds ($n=3$), this holds for $1$-forms ($p=1$). We thus define the Curl operator on $\Omega^1(M)$ as $\operatorname{Curl}_g := *_g d$. The connection between the spectral properties of $\operatorname{Curl}_g$ and the Hodge Laplacian $\Delta_1$ is established in the literature.

\begin{proposition}[\cite{Baer-curl, CapoferriVassiliev-BeyondHodge}] \label{prop:curl_properties}
	Let $(M, g)$ be a closed, oriented Riemannian $3$-manifold.
	\begin{enumerate}
		\item The operator $\operatorname{Curl}_g = *_g d$ acting on the space of coexact $1$-forms $\delta_g \Omega^2(M)$ is a self-adjoint operator with discrete spectrum accumulating at $\pm \infty$, and $0$ is not an eigenvalue \cite[Thm.~2.1]{CapoferriVassiliev-BeyondHodge}.
		
		\item On the space of coexact $1$-forms, the Hodge Laplacian satisfies $\Delta_1 = \operatorname{Curl}_g^2$ \cite[Lem.~2.3]{Baer-curl}.
		
		\item Consequently, a coexact $1$-form $\alpha$ is an eigenform of $\Delta_1$ with eigenvalue $\lambda$ if and only if $\alpha$ is a linear combination of eigenforms of $\operatorname{Curl}_g$ whose eigenvalues $\mu$ satisfy $\mu^2 = \lambda$ \cite[Lem.~2.3]{Baer-curl}.
	\end{enumerate}
\end{proposition}
{The spectrum of the Curl operator has been computed on Berger $3$-spheres ($b=c$) by Gibbons \cite[Eq. (5.11)]{Gibbons-SpectralAsymmetry} in a physics context and analyzed by Capoferri and Vassiliev \cite[Thm. D.1]{CapoferriVassiliev-BeyondHodge}. The authors cannot be sure whether this was the inspiration for the AI, as the model did not provide citations and its internal `Chain of Thought' log was not insightful. 
}

To bound the non-zero eigenvalues of $\Delta_1$ on coexact $1$-forms from below, it is strictly sufficient to bound the absolute value of the eigenvalues of $\operatorname{Curl}_g$.
Write $g_0 = g_{(1,1,1)}$ for the round metric on $\SU(2)$ or $\SO(3)$ with constant sectional curvature $1$.
%
Recall that the left-invariant $1$-forms dual to the standard basis $(E_1, E_2, E_3)$ of $\su(2)$ are denoted by $E_1^*, E_2^*, E_3^*$, where $dE_1^* = -2 E_2^* \wedge E_3^*$ (and cyclically). For the metric $g = g_{(a,b,c)}$, the basis $(\frac{1}{a}E_1^*, \frac{1}{b}E_2^*, \frac{1}{c}E_3^*)$ is orthonormal, and the Riemannian volume form is $\mu_g = \frac{1}{abc} E_1^* \wedge E_2^* \wedge E_3^*$.

We define a diagonal endomorphism $T: \Omega^1(\SU(2)) \to \Omega^1(\SU(2))$ acting on the left-invariant frame as
\begin{equation} \label{eq:operator_T}
	T(E_1^*) = a^2 E_1^*, \quad T(E_2^*) = b^2 E_2^*, \quad T(E_3^*) = c^2 E_3^*.
\end{equation}

\begin{lemma}[Transformation of $\delta_g$ and $\operatorname{Curl}_g$] \label{lem:operator_transform}
	For any $1$-form $\alpha$, we have the following relations:
	\begin{align}
		\delta_g \alpha &= \delta_0 (T \alpha), \label{eq:delta_transform} \\
		\operatorname{Curl}_g \alpha &= abc\, T^{-1} \operatorname{Curl}_0 \alpha, \label{eq:curl_transform}
	\end{align}
	where $\delta_0$ and $\operatorname{Curl}_0$ denote the codifferential and Curl operator with respect to 
	$g_0$.
\end{lemma}

\begin{proof}
	We explicitly compute the action of the Hodge star $*_g$ on the basis of left-invariant forms.
	Since the basis $\left(\frac{1}{a}E_1^*, \frac{1}{b}E_2^*, \frac{1}{c}E_3^*\right)$ is $g$-orthonormal, we have
	\[
	*_g \left(\frac{1}{a}E_1^*\right) = \left(\frac{1}{b}E_2^*\right) \wedge \left(\frac{1}{c}E_3^*\right) \quad \Longrightarrow \quad *_g(E_1^*) = \frac{a}{bc} E_2^* \wedge E_3^* = \frac{a^2}{abc} E_2^* \wedge E_3^*.
	\]
	For the round metric $g_0$ (where $a=b=c=1$), we simply have $*_0(E_1^*) = E_2^* \wedge E_3^*$. 
	
	Recalling the definition of the operator $T$, we find $\frac{1}{abc} *_0 (T E_1^*) = \frac{a^2}{abc} *_0(E_1^*) = \frac{a^2}{abc} E_2^* \wedge E_3^*$. By linearity, this establishes the following relation for any $1$-form $\alpha$:
	\begin{equation} \label{eq:star_1form}
		*_g \alpha = \frac{1}{abc} *_0 (T \alpha).
	\end{equation}
	Similarly, for $2$-forms, applying $*_g$ to the orthonormal basis element yields
	\[
	*_g \left(\frac{1}{bc} E_2^* \wedge E_3^*\right) = \frac{1}{a} E_1^* \quad \Longrightarrow \quad *_g (E_2^* \wedge E_3^*) = \frac{bc}{a} E_1^* = \frac{abc}{a^2} E_1^*.
	\]
	Since $*_0(E_2^* \wedge E_3^*) = E_1^*$ and $T^{-1}(E_1^*) = \frac{1}{a^2} E_1^*$, we deduce 
	for any $2$-form $\omega$ that
	\begin{equation} \label{eq:star_2form}
		*_g \omega = abc\, T^{-1} (*_0 \omega).
	\end{equation}
Furthermore, $*_g = abc *_0$ on $3$-forms.
		
We now use \eqref{eq:star_1form} 
to evaluate the codifferential in dimension $3$. 
For $\alpha$ a $1$-form, we get
\begin{align*}
\delta_g \alpha &
 = -*_gd*_g\alpha
= -*_g d \left( \frac{1}{abc} *_0 (T \alpha) \right) 
\\ & 
= - \left(abc *_0\right) \left( \frac{1}{abc} d *_0 (T \alpha) \right) 
= -*_0 d *_0 (T \alpha) = \delta_0(T\alpha).
\end{align*}
	For the Curl operator $\operatorname{Curl}_g \alpha = *_g (d\alpha)$, applying the rule for $2$-forms \eqref{eq:star_2form} immediately yields
	\[
	\operatorname{Curl}_g \alpha = abc\, T^{-1} (*_0 d\alpha) = abc\, T^{-1} \operatorname{Curl}_0 \alpha. \qedhere
	\]
\end{proof}
\begin{remark}\label{rem:curl operator vs weitzenboeck}
	The transformation operator $abc\, T^{-1}$ in \eqref{eq:curl_transform} exhibits a direct correspondence to the constant fiber-wise tensor computed via the Weitzenböck formula in Section \ref{sec:Hodge SU(2)}. 
	
	Recall from Proposition \ref{prop:Hodge-Laplace_on_k_rep} that the representation-independent part of the Hodge Laplacian is given by the diagonal matrix:
	\[
	-C_{\nabla}^{\su(2)} + 2q(R) = 4\diag\left(\frac{b^{2}c^{2}}{a^{2}},\frac{a^{2}c^{2}}{b^{2}},\frac{a^{2}b^{2}}{c^{2}}\right).
	\]
	Since $\operatorname{Curl}_0$ acts on the left-invariant basis forms as $\operatorname{Curl}_0(E_1^*) = -2 E_1^*$, the transformed operator yields $\operatorname{Curl}_g(E_1^*) = abc\, T^{-1}(-2 E_1^*) = -2\frac{bc}{a}E_1^*$. Squaring this relation recovers exactly $4\frac{b^2c^2}{a^2}$. 
	Thus, the transformation matrix $(abc\, T^{-1})^2$ governing the Curl operator is, up to a constant factor of $4$, identical to the geometric tensor $-C_{\nabla}^{\su(2)} + 2q(R)$.
\end{remark}

\subsection{Bounding the coexact eigenvalues}
We now establish the sharp lower bound for the eigenvalues on coexact 1-forms using the factorisation of the Hodge Laplacian and the transformation properties of the Curl operator.
\begin{proposition}\label{prop:coclosed_bound_proof}
Let $\Delta_1$ denote the Hodge-Laplace operator on $1$-forms associated to $\SU(2)$ endowed with $g_{\abc}$, and let $\alpha$ be a $1$-form on $\SU(2)$.  
If $\Delta_1 \alpha=\lambda\alpha$, then 
$$
\lambda \geq 
\min\left(4\frac{b^2c^2}{a^2}, 4\frac{a^2c^2}{b^2}, 4\frac{a^2b^2}{c^2}\right)
.
$$
Moreover, equality is attained by one of the left-invariant forms $E_1^*$, $E_2^*$, or $E_3^*$. 
%
\end{proposition}

\begin{proof}
We assume without loss of generality that $a \ge b \ge c > 0$.
By Proposition \ref{prop:curl_properties}, we may also assume that $\alpha \neq 0$ is an eigenform of $\operatorname{Curl}_g$ with $\operatorname{Curl}_g \alpha = \mu \alpha$, and $\lambda = \mu^2$. Our goal is to show $|\mu| \ge 2\frac{bc}{a}$.
	
	We set $\tilde{\lambda} := \frac{\mu}{abc}$ and define the $1$-form $\beta = T\alpha$. By Lemma \ref{lem:operator_transform}, since $\delta_g \alpha = 0$, we have $\delta_0 \beta = 0$, meaning $\beta$ is a coclosed $1$-form with respect to the round metric $g_0$. Substituting into \eqref{eq:curl_transform} yields
	\begin{equation}
		\operatorname{Curl}_0 \alpha = \tilde{\lambda} (T\alpha) = \tilde{\lambda}\beta.
		\label{eq:C0_alpha}
	\end{equation}
	
	Let $P_0$ be the $L^2(g_0)$-orthogonal projection onto $\ker \delta_0$. We project $\alpha$ by setting $\psi := P_0(\alpha)$. By the Hodge decomposition for $g_0$, $\alpha = \psi + d\phi$ for some function $\phi$. Since the Curl operator vanishes on exact forms, $\operatorname{Curl}_0 \psi = \operatorname{Curl}_0 \alpha = \tilde{\lambda} \beta$.
	
	We now utilize the known spectrum of $\operatorname{Curl}_0$.
	For $a=b=c=1$, the non-zero absolute eigenvalues of $\operatorname{Curl}_0$ on coexact forms are explicitly given by $k+2$ for integers $k \ge 0$. Thus, the smallest absolute non-zero eigenvalue of $\operatorname{Curl}_0$ on $\ker \delta_0$ is exactly $2$, yielding the spectral gap inequality $\|\operatorname{Curl}_0\psi\|_0 \ge 2\|\psi\|_0$, where $\|\cdot\|_0$ denotes the $L^2(g_0)$-norm. Taking the norm of $\operatorname{Curl}_0 \psi = \tilde{\lambda} \beta$ gives
	\begin{equation}
		|\tilde{\lambda}|\,\|\beta\|_0 = \|\operatorname{Curl}_0\psi\|_0 \ge 2\|\psi\|_0.
		\label{eq:spectral_gap_ineq}
	\end{equation}
	
	On the other hand, since $\beta \in \ker \delta_0$, it is $L^2(g_0)$-orthogonal to the exact form $d\phi$. Hence,
	\[
	\langle \psi,\beta\rangle_0 = \langle \alpha - d\phi,\beta\rangle_0 = \langle \alpha,\beta\rangle_0 = \langle T^{-1}\beta,\beta\rangle_0.
	\]
	Since we assumed $a \ge b \ge c$, the diagonal operator $T^{-1} = \diag(a^{-2},b^{-2},c^{-2})$ satisfies $T^{-1} \ge a^{-2}\Id$ in the sense of symmetric operators. Therefore, $\langle T^{-1}\beta,\beta\rangle_0 \ge a^{-2}\|\beta\|_0^2$. Applying the Cauchy-Schwarz inequality to the left side, we get
	\[
	\|\psi\|_0\,\|\beta\|_0 \ge \langle \psi,\beta\rangle_0 \ge a^{-2}\|\beta\|_0^2 \quad \implies \quad \|\psi\|_0 \ge a^{-2}\|\beta\|_0.
	\]
	Combining this with \eqref{eq:spectral_gap_ineq} yields
	\[
	|\tilde{\lambda}|\,\|\beta\|_0 \ge 2\|\psi\|_0 \ge 2a^{-2}\|\beta\|_0 \quad \implies \quad |\tilde{\lambda}| \ge \frac{2}{a^2}.
	\]
	Recalling that $\mu = \tilde{\lambda} abc$, we obtain $|\mu| = abc\,|\tilde{\lambda}| \ge 2\frac{bc}{a}$. Thus, $\lambda = \mu^2 \ge 4\frac{b^2c^2}{a^2}$.
	
To see that this bound is sharp, we refer back to our explicit computations in Section \ref{sec:Hodge SU(2)}. The $k=0$ isotypic component corresponds precisely to the $3$-dimensional space of left-invariant $1$-forms, which is spanned by $E_1^*, E_2^*$, and $E_3^*$. Since the first Betti number of $\SU(2)$ vanishes, these non-zero left-invariant forms are coexact. As already established in \eqref{eq:spec_k0}, the eigenvalues of $\Delta_1$ on this subspace are exactly $4\frac{b^2c^2}{a^2}$, $4\frac{a^2c^2}{b^2}$, and $4\frac{a^2b^2}{c^2}$. 
\end{proof}

\subsection{Conclusion of the proof}

We are now ready to establish the First Eigenvalue Conjecture fully for both $\SU(2)$ and $\SO(3)$.

\begin{theorem}\label{thm:first eigenvalue}
The smallest non-zero eigenvalue of the Hodge-Laplacian on $1$-forms for $g_{(a,b,c)}$ is given by
	\begin{align*}
		\lambda_1(\SU(2), g_{(a,b,c)}) &= \min\left( \frac{4b^2c^2}{a^2}, \frac{4a^2c^2}{b^2}, \frac{4a^2b^2}{c^2}, a^2+b^2+c^2 \right), \\
		\lambda_1(\SO(3), g_{(a,b,c)}) &= \min\left( \frac{4b^2c^2}{a^2}, \frac{4a^2c^2}{b^2}, \frac{4a^2b^2}{c^2} \right).
	\end{align*}
\end{theorem}

\begin{proof}
	Because $H^1(G)=0$ for $G \in \{\SU(2), \SO(3)\}$, the non-zero spectrum of $\Delta_1$ is the disjoint union of the spectra on exact $1$-forms and coexact $1$-forms. Furthermore, the exact spectrum of $\Delta_1$ strictly coincides with the non-zero spectrum of the Laplace-Beltrami operator $\Delta_0$ on functions, as the exterior derivative $d$ intertwines the operators ($\Delta_1 d = d \Delta_0$).
	
	For the \textit{coexact part}, Proposition \ref{prop:coclosed_bound_proof} establishes that the minimum eigenvalue is strictly given by $\min\left(4\frac{b^2c^2}{a^2}, 4\frac{a^2c^2}{b^2}, 4\frac{a^2b^2}{c^2}\right)$.
	
	It remains to compare this with the \textit{exact part} (the spectrum of $\Delta_0$):
	\begin{itemize}

\item \textit{On $\SU(2)$:} 
	The smallest positive eigenvalue of $\Delta_0$ is given by (see \cite[Thm.~1.1]{Lauret-SpecSU(2)})
	$$
	\min\left(a^2+b^2+c^2, 4(a^2+b^2), 4(a^2+c^2), 4(b^2+c^2)\right).
	$$
	One can easily check that 
	\begin{equation}\label{eq:min(Delta_1^0)<=min(Delta_0^2)}
	\min\left( \frac{4b^2c^2}{a^2}, \frac{4a^2c^2}{b^2}, \frac{4a^2b^2}{c^2} \right)
	\leq 
	\min\left(4(a^2+b^2), 4(a^2+c^2), 4(b^2+c^2)\right),
	\end{equation}
	thus the overall first eigenvalue of $\Delta_1$ is the one stated. 
%

%

\item \textit{On $\SO(3)$:} 
	The smallest positive eigenvalue of $\Delta_0$ is given by (see \cite[Thm.~1.2]{Lauret-SpecSU(2)})
	$$
	\min\left(4(a^2+b^2), 4(a^2+c^2), 4(b^2+c^2)\right),
	$$
	so the stated formula follows again from \eqref{eq:min(Delta_1^0)<=min(Delta_0^2)}.

	\end{itemize}
	This completes the proof.
\end{proof}

\section{Spectral inverse result}
In this section, we investigate the inverse spectral problem for the Hodge-Laplace operators on homogeneous 3-spheres and the projective space $\SO(3)$. We focus on the spectrum of $\Delta_1$ acting on 1-forms. Note that due to the Hodge star isomorphism and the commutation with the exterior derivative, the spectrum of $\Delta_1$ coincides with the spectrum of $\Delta_2$ and contains the spectrum of the Laplacian on functions $\Delta_0$ (excluding zero). Thus, $\Delta_1$ contains the most comprehensive spectral information.

Furthermore, leveraging the explicit formula for the first eigenvalue established in Theorem \ref{thm:first eigenvalue}, we prove that the spectrum of $\Delta_1$ determines the isometry class of the metric $g_{(a,b,c)}$ unconditionally.

\subsection{Spectral invariants}
We show that for homogeneous 3-manifolds, several fundamental geometric invariants are indeed spectrally determined.
	\begin{proposition}
		Let \((M^{3},g)\) be a three-dimensional homogeneous Riemannian manifold.  
Then 
	the volume \(\operatorname{vol}(M,g)\), 
	the scalar curvature \(\mathrm{Scal}\), 
	the norm of the full curvature tensor \(\|R\|\), 
	and the norm of the Ricci tensor \(\|\operatorname{Ric}\|\) 
are each spectral invariants of both
		the Hodge--Laplacian and the Laplace--Beltrami operator.
	\end{proposition}

	\begin{proof}
		We use \cite{BransonGilkeyOrsted-HeatInv} to find the heat-invariants for functions or $1$-forms on 3-dimensional manifolds:
		\begin{align*}
			a_{0}\propto \vol,\quad a_{1}\propto \Scal^{2},
		\end{align*}
and the coefficient $a_{2}$ is a linear combination of $a_{0},a_{1}$, $\|R\|^2$ and $\|\Ric\|^{2}$. More precisely,
		\begin{align*}
			&a_{2}\propto \frac{5}{2} \Scal^2 - \|\Ric\|^2 + \|R\|^2\quad \text{ for functions},\\
			&a_{2}\propto -\frac{45}{2} \Scal^2 + 87 \|\Ric\|^2 - 12 \|R\|^2\quad \text{ for $1$-forms}.
		\end{align*} 
		
On a $3$-dimensional manifold we have the following orthogonal decomposition of the curvature tensor which is obtained as a special case of the orthogonal decomposition described in \cite[1.114]{Besse-Einstein} to $n=3$ done in \cite[1.119]{Besse-Einstein}:
		\begin{align*}
			R=\frac{\Scal}{12} g\owedge g+\left(\Ric -\frac{\Scal}{3}g\right)\owedge g,
		\end{align*}
		where $\owedge$ denotes the Kulkarni-Nomizu product. We use formulas of the squared length for Kulkarni-Nomizu products stated in \cite[Lem. 7.22]{Lee-IRM} as well as $\langle\Ric,g\rangle =\Scal$ and $\langle g,g\rangle=3$ to conclude that 
		\begin{align*}
			\|R\|^{2}&=\left\|\frac{\Scal}{12} g\owedge g\right\|^{2}+\left\|\left(\Ric -\frac{\Scal}{3}g\right)\owedge g\right\|^{2}\\
			&=\frac{\Scal^{2}}{144}\cdot\left(4{\|g\|^{2}}+4\cdot \tr(g)^{2}\right)+4\cdot \left\|\Ric -\frac{\Scal}{3}g\right\|^{2}+4\cdot \tr\left(\Ric -\frac{\Scal}{3}g\right)^{2}\\
			&=\frac{\Scal^{2}}{144}\cdot 48+4\cdot\left(\|\Ric\|^{2}-\frac{\Scal^{2}}{3}\right)\\
			&=4\|\Ric\|^{2}-\Scal^{2}.
		\end{align*}
		{This formula coincides with \cite{Wikipedia-RicciDecomposition} as the Weyl-tensor vanishes in dimension $3$.}
		As $a_{2}$ and $\Scal$ are spectral invariants, plugging $\|R\|^{2}$ into $a_{2}$, yields that $\|\Ric\|^{2}$ and hence $\|R\|^{2}$ are spectral invariants.
	\end{proof}
	\subsection{Isospectrality results}
	Based on the spectral invariants derived above, we can now address the inverse spectral problem. We first consider the case of non-positive scalar curvature, where the heat invariants alone provide sufficient information.
	\begin{corollary}\label{cor:non-positive scalar curvature isospectral}
	Within the class of homogeneous metrics on $S^3$ or $\SO(3)$ with non-positive scalar curvature, the isometry class is uniquely determined by the spectrum of the Hodge-Laplacian on $p$-forms for any single degree $p \in \{0, 1, 2, 3\}$.
	\end{corollary}
	\begin{proof}
		Kling and Schueth \cite{KlingSchueth-DiracSpec} investigated the first eigenvalue of the Dirac operator on $\SU(2)$, $\SO(3)$ and they found that in the non-positive scalar curvature case, the following determine the parameters $a,b,c$ up to permutation uniquely: 
		\begin{itemize}
			\item The volume, $\vol$,
			\item the scalar curvature, $\Scal$,
			\item the value $5\Scal^{2}-8\|\Ric\|^{2}-7\|R\|^{2}$.
		\end{itemize}
		However, by the last proposition, the geometric properties $\vol$, $\|R\|^{2}$, $\|\Ric\|^{2}$ and $\Scal^{2}$ are themselves already spectral invariants.
	\end{proof}

In the arbitrary case, Lin, Schmidt and Sutton (see \cite[Thm.~1.6]{LinSchmidtSuttonII}) proved that the first three heat invariants do not determine the metric parameters. 
However, by incorporating the information from the first eigenvalue, we obtain the following result.
\begin{theorem}\label{thm:spectral_determination_positive}
	The isometry class of $(S^{3},g_{(a,b,c)})$ or $(\SO(3),g_{(a,b,c)})$ is uniquely determined by the spectrum of the Hodge-Laplacian on $p$-forms for any single degree $p \in \{0, 1, 2, 3\}$.
\end{theorem}

\begin{proof}
	The spectrum of $\Delta_1$ determines the heat invariants $a_0$ and $a_1$, which in turn determine the volume and the scalar curvature $\Scal$. Following the notation in \cite{KlingSchueth-DiracSpec}, let $x=a^2, y=b^2, z=c^2$ and let $\sigma_1, \sigma_2, \sigma_3$ denote their elementary symmetric polynomials, that is, 
	\[ \sigma_1 = x+y+z, \quad \sigma_2 = xy+yz+zx, \quad \sigma_3 = xyz. \]
	The volume fixes the product $\sigma_3$.
	
	By the explicit formula for the first eigenvalue established in Theorem \ref{thm:first eigenvalue}, $\lambda_1$ provides a third independent condition. We distinguish two cases:
	
	\begin{enumerate}
		\item The minimum is attained by the exact forms, i.e., $\lambda_1 = x+y+z$.
		In this case, the spectrum determines $\sigma_1 = \lambda_1$. The scalar curvature is given by (see \cite[(17)]{KlingSchueth-DiracSpec})
		\[ \Scal = 8\sigma_1 - 2\frac{\sigma_2^2}{\sigma_3}.\quad  \]
		Since $\sigma_1, \sigma_3$ and $\Scal$ are known, this equation uniquely determines $\sigma_2$ (as the positive solution of a quadratic equation). The metric parameters $x,y,z$ are then the unique roots of the polynomial $P(t) = t^3 - \sigma_1 t^2 + \sigma_2 t - \sigma_3$, by Vieta's formulas (see for example \cite{Weisstein-VietaFormulas}).
		
		\item The minimum is attained by a coclosed form, e.g., $\lambda_1 = 4yz/x$.
		Using $\sigma_3 = xyz$, we can rewrite this as $\lambda_1 = 4\sigma_3 / x^2$. Thus, $\lambda_1$ and $\sigma_3$ uniquely determine one parameter, say $x$. With $x$ known, the product of the remaining parameters is fixed by $yz = \sigma_3/x$. The scalar curvature equation then becomes a symmetric relation in $y$ and $z$, which, together with their product, uniquely determines $y$ and $z$.
	\end{enumerate}
	
	In both cases, the spectral data allows us to recover the parameters $a,b,c$ up to permutation. Thus, the isometry class is spectrally determined.
\end{proof}

	\bibliographystyle{plain}

\begin{thebibliography}{BLP22b}


		\bibitem[AH25]{AgricolaHenkel}
		{\sc I. Agricola, J. Henkel}.
		{\it The Laplace-Beltrami spectrum on naturally reductive homogeneous spaces}.
		Preprint, 2025.
		arXiv:2503.21416.
		DOI: \href{https://doi.org/10.48550/arXiv.2503.21416}{10.48550/arXiv.2503.21416}.
		

	
	\bibitem[Bä19]{Baer-curl}
	{\sc C. B\"ar}.
	{\it The curl operator on odd-dimensional manifolds}.
	J. Math. Phys. \textbf{60}:3 (2019), 031501.
	DOI: \href{https://doi.org/10.1063/1.5082528}{10.1063/1.5082528}.
	arXiv: \href{https://arxiv.org/abs/1702.02044}{1702.02044}.
		
		
		\bibitem[Be08]{BenHalima-Grassmannian}
		{\sc M. Ben Halima}.
		{\it Spectrum of the Hodge Laplacian on complex Grassmannian $\mathrm{Gr}_2(\mathbb{C}^{m+2})$.}
		Bull. Sci. Math. \textbf{132}:1 (2008), 19--26.
		DOI: \href{http://dx.doi.org/10.1016/j.bulsci.2007.04.003}{10.1016/j.bulsci.2007.04.003}.
			
			\bibitem[BB82]{BerardBergeryBourguignon-LaplSubmersion}
		{\sc L. B\'erard-Bergery, J.P. Bourguignon}.
		{\it Laplacians and Riemannian submersions with totally geodesic fibres}.
		Illinois J. Math. \textbf{26} (1982), 181--200.
		DOI: \href{https://doi.org/10.1215/ijm/1256046790}{10.1215/ijm/1256046790}.

\bibitem[Be]{Besse-Einstein}
	{\sc A. Besse}.
	{Einstein manifolds.}
Reprint of the 1987 edition. Classics in Mathematics,
	Springer, Berlin, 2008.
	DOI: \href{https://doi.org/10.1007/978-3-540-74311-8}{10.1007/978-3-540-74311-8}.

\bibitem[BP13]{BettiolPiccione13a}
	{\sc R. Bettiol, P. Piccione}.
	{\it Bifurcation and local rigidity of homogeneous solutions to the {Y}amabe problem on spheres.}
	Calc. Var. Partial Differential Equations \textbf{47}:3--4 (2013), 789--807.
	DOI: \href{http://dx.doi.org/10.1007/s00526-012-0535-y} {10.1007/s00526-012-0535-y}.
		
\bibitem[BLP22a]{BLPhomospheres}
	{\sc R. Bettiol, E.A. Lauret, P. Piccione}.
	{\it The first eigenvalue of a homogeneous CROSS.}
	J. Geom. Anal. \textbf{32} (2022), 76.
	DOI: \href{https://doi.org/10.1007/s12220-021-00826-7} {10.1007/s12220-021-00826-7}.
	
\bibitem[BLP22b]{BLPfullspec}
	{\sc R. Bettiol, E.A. Lauret, P. Piccione}.
	{\it Full Laplace spectrum of distance spheres in symmetric spaces of rank one.}
	Bull. Lond. Math. Soc. \textbf{54}:5 (2022), 1683--1704.
	DOI: \href{http://dx.doi.org/10.1112/blms.12650} {10.1112/blms.12650}.
		
	
		\bibitem[BG\O90]{BransonGilkeyOrsted-HeatInv}
		{\sc T.P. Branson, P.B. Gilkey, B. \O rsted}.
		{\it Leading terms in the heat invariants for the Laplacians of the
			{de Rham}, signature, and spin complexes}.
		Math. Scand. \textbf{66} (1990), 307--319.
		DOI: \href{http://dx.doi.org/10.7146/math.scand.a-12314}{10.7146/math.scand.a-12314}.
		
		\bibitem[Ba92]{Baer-DiracLensSpaces}
		{\sc C. B\"ar}.
		{\it The Dirac operator on homogeneous spaces and its spectrum on 3-dimensional lens spaces}.
		Arch. Math. (Basel) \textbf{59} (1992), 65--79.
		DOI: \href{https://doi.org/10.1007/BF01199016}{10.1007/BF01199016}.\\
		\emph{Erratum:} August 1997, \url{https://www.math.uni-potsdam.de/fileadmin/user_upload/Prof-Geometrie/Dokumente/Publikationen/erratum.pdf}.
		
		
		\bibitem[CV26]{CapoferriVassiliev-BeyondHodge}
		{\sc M. Capoferri, D. Vassiliev}.
		{\it Beyond the Hodge Theorem: curl and asymmetric pseudodifferential projections}.
		J. Lond. Math. Soc. \textbf{113}:1 (2026), e70431.
		DOI: \href{https://doi.org/10.1112/jlms.70431}{10.1112/jlms.70431}.
		
		
		\bibitem[CC90]{Colbois}
		{\sc B. Colbois, G. Courtois}.
		{\it A note on the first nonzero eigenvalue of the Laplacian acting on $p$-forms}.
		Manuscripta Math. \textbf{68}:2 (1990), 143--160.
		DOI: \href{https://doi.org/10.1007/BF02568757}{10.1007/BF02568757}.
		



\bibitem[ElC04]{Chami04}
	{\sc F. {El Chami}}.
	{\it Spectra of the Laplace operator on Grassmann manifolds}.
	Int. J. Pure Appl. Math. \textbf{12}:4 (2004), 395--417. 

\bibitem[ElC12]{Chami12}
	{\sc F. {El Chami}}.
	{\it A branching law from {${\rm Sp}(n)$} to {${\rm Sp}(q)\times {\rm Sp}(n-q)$} and an application to {L}aplace operator spectra}.
	Indian J. Pure Appl. Math. \textbf{43}:1 (2012), 71--86. 
	DOI: \href{http://dx.doi.org/10.1007/s13226-012-0005-4} {10.1007/s13226-012-0005-4}.
	
	\bibitem[Gi80]{Gibbons-SpectralAsymmetry}
	{\sc G.W. Gibbons}.
	{\it Spectral asymmetry and quantum field theory in curved spacetime}.
	Ann. Physics \textbf{125}:1 (1980), 98--116.
	DOI: \href{https://doi.org/10.1016/0003-4916(80)90120-7}{10.1016/0003-4916(80)90120-7}.

\bibitem[GH16]{GierHislop-Hodge5}
		{\sc M.E. Gier, P.D. Hislop}.
		{\it The Multiplicity of Eigenvalues of the Hodge Laplacian on 5-Dimensional Compact Manifolds}.
		J. Geom. Anal. \textbf{26}:4 (2016), 3176--3193.
		DOI: \href{https://doi.org/10.1007/s12220-015-9666-7}{10.1007/s12220-015-9666-7}.

\bibitem[GLP]{GilkeyLeahyPark-book}
	{\sc P.B. {Gilkey}, J.V. {Leahy}, J. {Park}}.
	{Spectral geometry, Riemannian submersions, and the Gromov-Lawson conjecture}.
	{\it Stud. Adv. Math.}, Chapman \& Hall/CRC, Boca
	Raton FL, 1999.
		
		
		\bibitem[GP96]{GilkeyPark-RiemSubmersion}
		{\sc P.B. Gilkey, J.H. Park}.
		{\it Riemannian submersions which preserve the eigenforms of the Laplacian}.
		Illinois J. Math. \textbf{40} (1996), 194--201.
		DOI: \href{https://doi.org/10.1215/ijm/1255986099}{10.1215/ijm/1255986099}.
		
		
\bibitem[GM06]{GornetMcGowan06}
    {\sc R. Gornet, J. McGowan}.
    {\it Lens Spaces, isospectral on forms but not on functions}.
    LMS J. Comput. Math. \textbf{9} (2006), 270--286.
    DOI: \href{http://dx.doi.org/10.1112/S1461157000001273} {10.1112/S1461157000001273}.

\bibitem[GM20]{GornetMcGowan20}
    {\sc R. Gornet, J. McGowan}.
    {\it Corrigendum to ``Lens Spaces, isospectral on forms but not on functions''}.
    DOI: \href{
    https://doi.org/10.48550/arXiv.1906.07787} {10.48550/arXiv.1906.07787} (2020).

		
		\bibitem[He25]{Henkel-AI-Math}
		{\sc J. Henkel}.
		{\it The mathematician's assistant: integrating AI into research practice}.
		Mathematische Semesterberichte \textbf{72}:2 (2025), 117--144.
		DOI: \href{https://doi.org/10.1007/s00591-025-00400-0}{10.1007/s00591-025-00400-0}.
		arXiv: \href{https://arxiv.org/abs/2508.20236}{2508.20236}.
		
		
		\bibitem[He26]{HL2025Code}
		{\sc J. Henkel}.
		{\it Spectral Analysis of the Hodge-Laplacian on SU(2)}.
		Electronic Supplementary Material, GitHub repository, 2025.
		\url{https://github.com/jmhenkel/Code/blob/main/Spectral_Analysis_SU2.ipynb}.
		
%
%

\bibitem[Ik89]{Ikeda88}
    {\sc A. Ikeda}.
    {\it Riemannian manifolds $p$-isospectral but not $p+1$-isospectral}.
    In \textsl{Geometry of manifolds ({M}atsumoto, 1988)}, 383--417,
    Perspect. Math. \textbf{8}, 1989.
		
\bibitem[IT78]{IkedaTaniguchi-SnPnC}
	{\sc A. Ikeda, Y. Taniguchi}.
	{\it Spectra and eigenforms of the Laplacian on $S^n$ nd $P^n(\mathbb{C})$.}
	Osaka J. Math. \textbf{15} (1978), 515--546.


		\bibitem[KS22]{KlingSchueth-DiracSpec}
		{\sc J. Kling, D. Schueth}.
		{\it On the Dirac Spectrum of Homogeneous 3-Spheres}.
		J. Geom. Anal. \textbf{32} (2022), article 275.
		DOI: \href{https://doi.org/10.1007/s12220-022-00997-x}{10.1007/s12220-022-00997-x}.



\bibitem[La18]{Lauret-p-spectralens}
    {\sc E.A. Lauret}.
    {\it The spectrum on $p$-forms of a lens space}.
    Geom. Dedicata \textbf{197} (2018), 107--122.
    DOI: \href{http://dx.doi.org/10.1007/s10711-018-0322-9} {10.1007/s10711-018-0322-9}.

\bibitem[La19]{Lauret-SpecSU(2)}
	{\sc E.A. Lauret}.
	{\it The smallest Laplace eigenvalue of homogeneous 3-spheres}.
	Bull. Lond. Math. Soc. \textbf{51}:1 (2019), 49--69.
	DOI: \href{http://dx.doi.org/10.1112/blms.12213} {10.1112/blms.12213}.
		
\bibitem[LMR16]{LauretMiatelloRossetti-LensSpaces}
    {\sc E.A. Lauret, R.J. Miatello, J.P. Rossetti}.
    {\it Spectra of lens spaces from 1-norm spectra of congruence lattices.}
    Int. Math. Res. Not. IMRN \textbf{2016}:4 (2016), 1054--1089.
    DOI: \href{http://dx.doi.org/10.1093/imrn/rnv159} {10.1093/imrn/rnv159}.

\bibitem[Le]{Lee-IRM}
	{\sc J. M. Lee}.
	{Introduction to Riemannian Manifolds.}
	{\it Grad. Texts in Math.} \textbf{176}.
	Springer Cham, 2018.
	DOI: \href{https://doi.org/10.1007/978-3-319-91755-9} {10.1007/978-3-319-91755-9}.

\bibitem[LSS21]{LinSchmidtSuttonII}
	{\sc S. Lin, B. Schmidt, C.J. Sutton}.
	{\it Geometric structures and the {L}aplace spectrum, {P}art II}.
	Trans. Amer. Math. Soc. \textbf{374}:12 (2021), 8483--8530.
	DOI: \href{https://doi.org/10.1090/tran/8417} {10.1090/tran/8417}.

\bibitem[Lo12]{Lotay12}
	{\sc J.D. Lotay}.
	{\it Stability of coassociative conical singularities}.
	Comm. Anal. Geom. \textbf{20}:4 (2012), 803--867. 
	DOI: \href{https://doi.org/10.4310/CAG.2012.v20.n4.a5} {10.4310/CAG.2012.v20.n4.a5}.

\bibitem[Ma97]{Mashimo97}
	{\sc K. Mashimo}.
	{\it Spectra of the {L}aplacian on the {C}ayley projective plane.}
	Tsukuba J. Math. \textbf{21}:2 (1997), 367--396.
	DOI: \href{http://dx.doi.org/10.21099/tkbjm/1496163248} {10.21099/tkbjm/1496163248}.
	
	\bibitem[MM63]{MatsushimaMurakami-HarmonicForms}
	{\sc Y. Matsushima, S. Murakami}.
	{\it On vector bundle valued harmonic forms and automorphic forms on symmetric Riemannian manifolds}.
	Ann. of Math. (2) \textbf{78} (1963), 365--416.
	DOI: \href{https://doi.org/10.2307/1970348}{10.2307/1970348}.

\bibitem[MR09]{MR-survey}
    {\sc R.J. Miatello, J.P. Rossetti}.
    {\it Spectral properties of flat manifolds}.
    In \textsl{New developments in Lie theory and geometry}, 83--113,
    \textit{Contemp. Math.} \textbf{491}, Amer. Math. Soc., Providence, RI, 2009.

\bibitem[Mi76]{Milnor-Curvatures}
	{\sc J. Milnor}.
	{\it Curvatures of left invariant metrics on lie groups.}
	Adv. Math. \textbf{21}:3 (1976), 293--329.
	DOI: \href{http://dx.doi.org/10.1016/S0001-8708(76)80002-3} {10.1016/S0001-8708(76)80002-3}.
	
		

		
\bibitem[Ro]{Rosenberg-Laplacian}
	{\sc S. Rosenberg}.
	{The {L}aplacian on a {R}iemannian manifold}.
	\textit{London Math. Soc. Stud. Texts} \textbf{31}, 
	Cambridge University Press, Cambridge, 1997.
    DOI: \href{https://doi.org/10.1017/CBO9780511623783} {10.1017/CBO9780511623783}.	



		
		\bibitem[Se01]{Semmelmann-Habil}
		{\sc U. Semmelmann}.
		{\it Conformal Killing forms on Riemannian manifolds}.
		Habilitationsschrift, Fakult\"at f\"ur Mathematik und Informatik,
		Ludwig-Maximilians-Universit\"at M\"unchen, 2001.
		Available at
		\href{https://www.igt.uni-stuttgart.de/dokumente/semmelmann/semmelmann_Publ/killing20_24a.pdf}%
		{https://www.igt.uni-stuttgart.de/dokumente/semmelmann/semmelmann\_Publ/killing20\_24a.pdf}.
		
\bibitem[SW02]{SemmelmannWeingart-QK}
	{\sc U. Semmelmann, G. Weingart}.
	{\it Vanishing theorems for quaternionic Kähler manifolds}.
	J. Reine Angew. Math. \textbf{544} (2002), 111--132.		
	DOI: \href{http://dx.doi.org/10.1515/crll.2002.019} {10.1515/crll.2002.019}.
		
		
\bibitem[St84]{Sthanumoorthy84}
    {\sc N. Sthanumoorthy}.
    {\it Spectra of de {R}ham {H}odge operator on {${\rm SO}(n+2)/{\rm SO}(2)\times {\rm SO}(n)$}.}
    Bull. Sci. Math. (2) \textbf{108}:3 (1984), 297--320.

\bibitem[Ta]{Takeuchi}
	{\sc M. Takeuchi}.
	Modern spherical functions. (Transl. from the Japanese by Toshinobu Nagura.)
	{\it Transl. Math. Monogr.} \textbf{135}.
	Amer. Math. Soc., Providence, 1994.

\bibitem[Ta79]{Tanno79}
	{\sc S. Tanno}.
	{\it The first eigenvalue of the {L}aplacian on spheres.}
	T\v ohoku Math. J. (2) \textbf{31}:2 (1979), 179--185. 
	DOI: \href{http://dx.doi.org/10.2748/tmj/1178229837} {10.2748/tmj/1178229837}.

\bibitem[Ts81]{Tsukamoto81}
    {\sc C. Tsukamoto}.
    {\it Spectra of Laplace-Beltrami operators on $\textrm{SO}(n+2)/\textrm{SO}(2)\times\textrm{SO}(n)$ and $\textrm{Sp}(n+1)/\textrm{Sp}(1)\times\textrm{Sp}(n)$}.
    Osaka J. Math. \textbf{18}:2 (1981), 407--426.
    DOI: \href{http://dx.doi.org/10.18910/8349} {10.18910/8349}.

\bibitem[Wa]{Wallach-HarmonicAnalysis}
    {\sc N. Wallach}.
    {\it Harmonic analysis on homogeneous spaces}.
    { Pure and Applied Mathematics} \textbf{19}.
    Marcel Dekker, Inc., New York, 1973.
    
    \bibitem[We02]{Weisstein-VietaFormulas}
    {\sc E.W. Weisstein}.
    {\it Vieta's Formulas}.
    In: MathWorld---A Wolfram Web Resource.
    \url{https://mathworld.wolfram.com/VietasFormulas.html}.
    
\bibitem[Wi25]{Wikipedia-RicciDecomposition}
{\sc Wikipedia contributors}.
{\it Ricci decomposition.}
Wikipedia, The Free Encyclopedia.
Revision of 10 September 2025, 08:39 UTC.
URL: \url{https://en.wikipedia.org/w/index.php?title=Ricci_decomposition&oldid=1310557434#Related_formulas}.
Accessed 31 March 2026.
    

\end{thebibliography}

\end{document}